\documentclass[11pt,a4paper]{article}
\usepackage[T1]{fontenc}
\usepackage[utf8]{inputenc}
\usepackage{amsmath,amssymb,amsthm,mathrsfs,bm}
\usepackage{mathtools}
\usepackage[margin=2.6cm]{geometry}
\usepackage[numbers,sort&compress]{natbib}
\usepackage[colorlinks=true,linkcolor=blue,citecolor=blue]{hyperref}

\newcommand{\R}{{\mathbb R}}
\newcommand{\N}{{\mathbb N}}
\newcommand{\E}{\mathbb{E}}
\newcommand{\PP}{\mathbb{P}}
\newcommand{\Ocal}{\mathcal{O}}
\newcommand{\Ecal}{\mathcal{E}}
\newcommand{\Lcal}{\mathcal{L}}
\newcommand{\Vcal}{\mathcal{V}}
\newcommand{\Gcal}{\mathcal{G}}
\newcommand{\Acal}{\mathcal{A}}
\newcommand{\Gbb}{\mathbb{G}}
\newcommand{\fU}{\mathfrak{U}}
\newcommand{\calH}{\mathcal{H}}
\newcommand{\Tcal}{\mathcal{T}}
\newcommand{\Tr}{\operatorname{Tr}}
\newcommand{\norm}[1]{\left\|#1\right\|}
\newcommand{\normH}[1]{\left\|#1\right\|_{H}}
\newcommand{\normV}[1]{\left\|#1\right\|_{V}}
\newcommand{\normLp}[2]{\left\|#1\right\|_{L^{#2}}}
\newcommand{\ip}[2]{\langle #1,\, #2 \rangle}
\newcommand{\inner}[2]{\langle #1,\, #2 \rangle_{H}}
\newcommand{\abs}[1]{\left\lvert #1 \right\rvert}

\theoremstyle{plain}
\newtheorem{theorem}{Theorem}[section]
\newtheorem{lemma}[theorem]{Lemma}
\newtheorem{proposition}[theorem]{Proposition}
\newtheorem{corollary}[theorem]{Corollary}
\theoremstyle{definition}
\newtheorem{definition}[theorem]{Definition}
\newtheorem{assumption}[theorem]{Assumption}
\newtheorem{remark}[theorem]{Remark}

\title{Strong Galerkin Approximation, Malliavin Regularity, and Blow-Up for a Mixed Local--Nonlocal Stochastic Wave Equation}

\author{Francisco Delgado-Vences\thanks{Universitat Aut\`onoma de Barcelona,
Departament de Matem\`atiques, Edifici C, Facultat de Ci\`encies, 08193
Bellaterra, Spain. \texttt{FranciscoJavier.Delgado@uab.cat}}
\and
Jose Julian Pavon-Espa\~nol\thanks{{\bf Corresponding author}. Facultad de Ciencias,
Universidad Nacional Aut\'onoma de M\'exico, Ciudad Universitaria, Ciudad de
M\'exico, Mexico. \texttt{julian.pavon2@ciencias.unam.mx}}}

\date{}

\begin{document}
\maketitle

\begin{abstract}
We investigate the dynamical behavior of a class of semilinear stochastic wave equations on a bounded smooth domain $\Ocal\subset\R^d$ driven by additive trace-class noise, where the elastic response is governed by a \emph{mixed local--nonlocal} operator $\Acal=-\theta\Delta+\beta(-\Delta)^s$ with $s\in(0,1)$. A fundamental challenge in this setting is that the local and nonlocal operators do not commute on bounded domains: the natural Dirichlet basis fails to diagonalize the restricted fractional Laplacian. Consequently, we first establish the \emph{strong} convergence of the resulting non-diagonal, dense Galerkin approximation scheme. Leveraging these uniform energy bounds, we rigorously derive the associated It\^{o} energy identity. In the defocusing regime ($\varepsilon = +1$), this strong approximation yields global well-posedness on the energy-subcritical range, providing a unique probabilistically strong solution in the energy space $V \times H$. Within this variational framework, we conduct an analysis of the Malliavin regularity, showing $(u(t), v(t)) \in \mathbb{D}^{1,2}(V) \times \mathbb{D}^{1,2}(H)$, and leverage fractional Sobolev embeddings to prove that the one-dimensional probability law of $u(t, x_0)$ is absolutely continuous via the Bouleau--Hirsch criterion. In stark contrast, for the focusing regime ($\varepsilon = -1$), we establish local well-posedness and prove a rigorous dichotomy: under a negativity condition on the initial energy, either pathwise explosion occurs with positive probability in finite time, or the energy norm possesses an infinite second moment before an explicit critical time $T^*$. Finally, we observe how the dense Galerkin interaction matrices pose unique structural challenges for spatial statistical inference.
\end{abstract}

\noindent\textbf{Keywords:} stochastic wave equation; mixed local--nonlocal
operator; restricted fractional Laplacian; Galerkin approximation;
well-posedness; blow-up.

\medskip

\section{Introduction}

Consider the fractional Laplacian $(-\Delta)^{s}$, with $s\in (0,1)$, as the
\textbf{nonlocal} linear operator on $L^2(\Ocal)$, where $\Ocal \subset \R^d$
is a bounded subset with smooth boundary, defined as (see
\cite{XuCaraballo2023,Leoni-2023,Molisca2016,DiNezza-Palatucci2012})
\begin{align}\label{frac_laplace}
  (- \Delta)^{s} u = -\frac{1}{2}\, C(d,s) \int_{\R^d}
  \frac{u(x+y)+u(x-y) -2u(x) }{|y|^{d+2s}}\, dy .
\end{align}
We will study the SPDE
\begin{align}\label{stochastic_local_nonlocal}
\begin{cases}
\displaystyle \frac{\partial^2}{\partial t^2}u -\theta \Delta u
  +\beta (-\Delta)^{s}u + \varepsilon\,|u|^pu
  = \sigma \sum_{k\in\N} \lambda_k^{-\gamma}\, h_k(x)\, dw_k(t),
  &\text{in } (0,+\infty)\times\Ocal,\\[2pt]
u = 0, &\text{on } (0,+\infty)\times (\R^d \setminus \Ocal),\\[2pt]
u(0,x)= U_0(x),\quad \dfrac{\partial}{\partial t}u(0,x)= V_0(x),
  &\text{in } \Ocal,
\end{cases}
\end{align}
where $\theta,\beta,\sigma>0$, $\gamma\ge 0$, $\varepsilon\in\{+1,-1\}$
distinguishes the \emph{defocusing} ($\varepsilon=+1$) from the
\emph{focusing} ($\varepsilon=-1$) nonlinearity, and $\{w_k,\ k\geq 1\}$ is a
collection of independent standard Brownian motions on a stochastic basis
$(\Omega,\mathcal{F},\{\mathcal{F}_t\}_{t\geq 0}, \PP)$. Here $\Delta$ is the
Dirichlet Laplacian on the bounded smooth domain $\Ocal\subset \R^d$; the
sequences $\lambda_k$ and $h_k$ satisfy $-\Delta h_k = \lambda_k h_k$ with
$h_k\in H^1_0(\Ocal)$, $\{h_k\}$ orthonormal in $L^2(\Ocal)$, and we write
$u_k(t):= \langle u(t,\cdot),h_k(\cdot)\rangle$. Functions in $H^1_0(\Ocal)$
are extended by zero to $\R^d$; with this convention
$H_0^1(\Ocal)\hookrightarrow \widetilde H^{s}(\Ocal)$ for every $s\in(0,1)$,
where $\widetilde H^s(\Ocal)=\{u\in H^s(\R^d): u=0 \text{ a.e.\ on }
\R^d\setminus\Ocal\}$.

Mixed local--nonlocal operators of the form
$-\theta\Delta+\beta(-\Delta)^{s}$ have received considerable attention in
recent years, both for their analytic structure
\cite{BiagiDipierroValdinociVecchi2021,BiagiDipierroValdinociVecchi2022,BiagiDipierroValdinociVecchi2023}
and as models of superposed Brownian and $2s$-stable transport. On the
stochastic side, parabolic equations driven by the restricted fractional
Laplacian have been analysed in \cite{XuCaraballo2023,Xu-Caraballo-2022},
while blow-up for stochastic hyperbolic equations goes back to
\cite{Chow2009blowup,Chow2011,DozziLopezMimbela2010,BaoYuan2016}; the
concavity method we use in the focusing case originates with
\cite{Kaplan1963,Glassey1973,Levine1974a,Levine1974b,KnopsLevinePayne1974,PayneSattinger1975}. The hyperbolic
counterpart, which is the object of the present work, has received much less
attention; for example  in \cite{CuiHongJiSun2024}, where similar approach (Galerkin method) is used applied to a wave equation driven by only the Laplacian operator and a non-linear polynomial part is discussed.

Two features separate \eqref{stochastic_local_nonlocal} from both its purely
local and its parabolic relatives, and they govern the technical shape of the
whole paper. The first is algebraic. On a bounded domain with the exterior
Dirichlet condition, the restricted fractional Laplacian does not commute
with the Dirichlet Laplacian, so the natural Galerkin basis $\{h_k\}$
associated with the classical Dirichlet Laplacian fails to diagonalise
$\Acal$, and the projected stiffness matrix
$A_N=\theta\Lambda_N+\beta M_N$ is in general dense. The spectral decoupling
into independent scalar oscillators --- the standard route for the classical
stochastic wave equation --- is therefore unavailable. The second feature is
analytic: the equation is hyperbolic and undamped, so no smoothing
compensates for the loss of the spectral representation, and the energy space
$\Vcal=V\times H$ is the natural, and essentially the only, space in which to
work.

The absence of simultaneous diagonalisation does not, however, obstruct a
uniform energy theory, and establishing this is the starting point of the
paper. We show that the projected matrices $A_N$ satisfy coercivity and form
bounds independent of $N$ (Lemma~\ref{lem:galerkin}), that the associated
wave groups converge strongly to the limiting group, and that the
corresponding stochastic convolutions converge strongly in the energy space.
Together these yield \emph{strong} convergence of the full nonlinear Galerkin
approximation, rather than merely weak or subsequential convergence
(Proposition~\ref{prop:strongconv}).

That strong convergence is what drives the defocusing regime
($\varepsilon=+1$). Combining it with a pathwise reduction to a deterministic
wave equation driven by the stochastic convolution, we obtain global
existence, pathwise uniqueness and continuity in the energy space of a
probabilistically strong solution on the energy-subcritical range
(Theorem~\ref{thm:global-dd}). Passing to the limit in the finite-dimensional
It\^o formula along the Galerkin scheme then gives the It\^o energy identity
for the system (Theorem~\ref{thm:energy-identity}): the expected energy grows
affinely in time, with a deterministic drift equal to half the trace of the
noise covariance, and the fluctuation around it is a martingale.

In the purely local case ($\beta=0$) this averaged energy evolution law is a
known structural invariant of the stochastic wave equation, and preserving it
at the discrete level has become a design principle for numerical schemes.
Hong, Hou and Sun \cite{HongHouSun2022} construct fully discrete methods for
nonlinear stochastic wave equations with multiplicative noise - combining
compact finite differences or an interior penalty discontinuous Galerkin
discretisation in space with a discrete gradient method and Pad\'e
approximation in time - which reproduce the discrete averaged energy law
exactly, and reproduce it pathwise almost surely when the noise is additive.
The identity established here is the mixed local--nonlocal analogue of that
law, derived from a non-diagonal Galerkin scheme rather than a spectral one.
The distinction is not cosmetic: because the stiffness matrices $A_N$ are
dense, the transfer of such structure-preserving discretisations to
\eqref{stochastic_local_nonlocal} is not automatic, and the identity proved
below is precisely what such a transfer would have to preserve.

Within the same variational framework we study the regularity of the
defocusing solution as a random field. Malliavin calculus is the natural
tool: we prove that the energy solution is Malliavin differentiable, with
$(u(t),v(t))\in\mathbb{D}^{1,2}(V)\times\mathbb{D}^{1,2}(H)$, and that its
Malliavin derivative satisfies a well-defined variational evolution equation
(Theorem~\ref{thm:malliavin_diff}). In dimension one we then exploit the
fractional Sobolev embedding into H\"older spaces and the Bouleau--Hirsch
criterion \cite{nualart2006malliavin,SanzSole2005} to show that the law of
the pointwise evaluation $u(t,x_0)$ is absolutely continuous with respect to
Lebesgue measure on $\R$ (Theorem~\ref{thm:density}). This is what makes
densities --- and hence likelihoods built from point observations ---
meaningful objects for the model.

The focusing regime ($\varepsilon=-1$) behaves in the opposite way: the
energy is no longer coercive and global existence fails. Blow-up for
stochastic hyperbolic equations goes back to
\cite{Chow2009blowup,Chow2011,DozziLopezMimbela2010,BaoYuan2016}, and the
concavity argument we use originates with
\cite{Kaplan1963,Glassey1973,Levine1974a,Levine1974b,KnopsLevinePayne1974,PayneSattinger1975}.
We first prove local well-posedness up to a maximal stopping time, with an
explicit localized lower bound on that time
(Theorem~\ref{thm:local-fixed}). Under a negativity condition on the initial
energy we then show that no solution with finite second moments can exist
beyond an explicit critical time $T^{*}$, and we quantify the algebraic
profile of the resulting second-moment blow-up (Theorem~\ref{thm:blowup},
Proposition~\ref{prop:profile}). The stochastic setting forces a distinction
that is absent from the deterministic theory, since non-existence within a
moment class does not by itself imply pathwise explosion; separating the two
possibilities is the content of the dichotomy we establish
(Theorem~\ref{thm:dichotomy}): either the solution explodes pathwise with
positive probability before $T^{*}$, or the energy norm has infinite second
moment on $[0,T^{*}]$.

The Galerkin structure is developed throughout with a view towards
statistical inference. Projecting the equation and applying Girsanov's
theorem yields a likelihood whose covariance matrices are dense precisely
because of the nonlocal term, which is the obstruction discussed in
Section~\ref{sec:inference}.

Finally, we observe that the non-commutation of the local and nonlocal operators has direct implications for spatial statistical inference: projecting the equation and applying Girsanov's theorem yields a likelihood whose covariance matrices are dense precisely because of the nonlocal interaction (Section~\ref{sec:inference}).

Throughout the paper we work under the following standing assumptions.

\begin{assumption}\label{ass:standing}
\emph{(i) (Noise.)} $\gamma>d/4$. By Weyl's law
$k=N(\lambda_k)=\frac{\omega_d \abs{\Ocal}}{(2\pi)^d }\lambda_k^{d/2}
(1+o(1))$ for $k$ sufficiently large, whence $\lambda_k =C_0 k^{2/d}(1+o(1))$
and the quantity $\kappa:=\sigma^2\sum_{k\ge1}\lambda_k^{-2\gamma}$ is
finite; equivalently
$Q=\sigma^2\sum_k\lambda_k^{-2\gamma}h_k\otimes h_k$ is trace class on
$H=L^2(\Ocal)$ with $\Tr Q=\kappa$.

\emph{(ii) (Strict subcriticality.)} $0<p<\dfrac{4}{d-2}$ if $d\ge3$, and
$0<p<\infty$ if $d\le2$. Thus $p+2<2^*$ and
$H^1_0(\Ocal)\hookrightarrow\hookrightarrow L^{p+2}(\Ocal)$
\emph{compactly}.

\emph{(iii) (Data.)} $U_0\in H^1_0(\Ocal)$ and $V_0\in L^2(\Ocal)$,
deterministic or $\mathcal F_0$-measurable with
$\E\bigl(\normLp{\nabla U_0}{2}^2+\normLp{V_0}{2}^2
+\normLp{U_0}{p+2}^{p+2}\bigr)<\infty$.
\end{assumption}

Two further ranges of the parameters will play a role. The
\emph{energy-subcritical} range for the exponent is
\begin{equation}\label{eq:local-range}
0<p\le\frac{2}{d-2}\quad(d\ge3),\qquad 0<p<\infty\quad(d\le2),
\end{equation}
under which $V=H^1_0(\Ocal)\hookrightarrow L^{2(p+1)}(\Ocal)$; and the
\emph{strengthened noise condition} is
\begin{equation}\label{eq:gamma-strong}
  \gamma\;>\;\frac d4+\frac\delta2 ,
  \qquad
  \delta:=d\Bigl(\frac12-\frac1{p+2}\Bigr)\in(0,1).
\end{equation}

The rest of the paper is organized as follows. Section~\ref{sec:setting} introduces the functional setting and the properties of the mixed local--nonlocal operator. In Section~\ref{sec:convolution}, we construct the stochastic convolution and analyze its regularity. Section~\ref{sec:galerkin} is devoted to the Galerkin projection scheme and its algebraic properties. In Section~\ref{sec:defocusing}, we prove the global well-posedness of the defocusing equation and derive the It\^{o} energy identity. Section~\ref{sec:malliavin} contains the Malliavin regularity analysis and the proof of the existence of the probability density. Section~\ref{sec:focusing} is dedicated to the focusing case and the second-moment blow-up. Finally, Section~\ref{sec:inference} discusses the spatial H\"{o}lder regularity of the paths and the challenges of the dense matrix structure for statistical inference.

\section{Functional setting}\label{sec:setting}

\subsection{Fractional Sobolev spaces and the Gagliardo norm}

We recall, for $s\in(0,1)$ the fractional Laplacian is the nonlocal operator
\begin{equation}\label{eq:frac_laplace_def}
(-\Delta)^s u(x)=-\tfrac12 C(d,s)\int_{\R^d}
\frac{u(x+y)+u(x-y)-2u(x)}{|y|^{d+2s}}\,dy,\qquad x\in\Ocal.
\end{equation}
The natural energy space for this operator involves the fractional Sobolev
space $H^s(\R^d)$, defined as the set of functions $u \in L^2(\R^d)$ such
that the Gagliardo semi-norm is finite:
\begin{equation}\label{eq:gagliardo_seminorm}
    [u]_{H^s(\R^d)}^2 = \frac{C(d,s)}{2} \int_{\R^d} \int_{\R^d}
    \frac{|u(x) - u(y)|^2}{|x-y|^{d+2s}}\, dx\, dy,
\end{equation}
with full norm $\|u\|_{H^s(\R^d)}^2 = \|u\|_{L^2(\R^d)}^2 +
[u]_{H^s(\R^d)}^2$. Due to the nonlocal nature of the operator and the
exterior Dirichlet condition $u=0$ on $\R^d \setminus \Ocal$, we work in the
restricted subspace
\begin{equation}\label{eq:fractional_subspace}
    \widetilde H^s(\Ocal) = \bigl\{ u \in H^s(\R^d) : u = 0
    \text{ a.e.\ in } \R^d \setminus \Ocal \bigr\}.
\end{equation}

Let $\Ocal\subset\R^d$ be a bounded domain with $C^{1,1}$ boundary and
$H:=L^2(\Ocal)$, with inner product $\ip{\cdot}{\cdot}$ and norm
$\norm{\cdot}$. Let $\{h_k\}_{k\ge1}\subset H^1_0(\Ocal)$ be the
$L^2$-orthonormal basis of Dirichlet eigenfunctions of the classical
Laplacian, $-\Delta h_k=\lambda_k h_k$, $0<\lambda_1\le\lambda_2\le\cdots$,
and identify every $u\in H^1_0(\Ocal)$ with its extension by zero to $\R^d$. For a more detailed discussion of the fractional Laplacian operator can be found in \cite{DiNezza-Palatucci2012}.

The object of the paper is the \emph{mixed local--nonlocal} operator
\begin{equation}\label{eq:mixed-operator}
\Acal:=-\theta\Delta+\beta(-\Delta)^s,\qquad \theta,\beta>0,
\end{equation}
whose energy form is the symmetric bilinear form
\begin{equation}\label{eq:frac-form}
\Acal(u,v):=\theta\ip{\nabla u}{\nabla v}+\beta\,\mathcal M_s(u,v),\qquad
\mathcal M_s(u,v)=\frac{C(d,s)}{2}\iint_{\R^d\times\R^d}
\frac{(u(x)-u(y))(v(x)-v(y))}{|x-y|^{d+2s}}\,dx\,dy ,
\end{equation}
so that $\ip{\Acal u}{v}=\Acal(u,v)$ for $u,v\in H^1_0(\Ocal)$ (extended by
zero), the pairing being the duality $[\cdot,\cdot]$ of $H^{-1}$ with
$H^1_0$; in particular $\mathcal M_s(u,u)=[u]^2_{H^s(\R^d)}\ge0$.

\subsection{Function spaces and the Gelfand triple}

We take $H:=L^2(\Ocal)$ as pivot space (identified with its dual) and
\begin{equation}\label{eq:V-def}
V:=D(\Acal^{1/2}),\qquad
\normV{u}^2:=\ip{\Acal u}{u}=\theta\norm{\nabla u}^2+\beta[u]^2_{H^s(\R^d)} .
\end{equation}

\begin{lemma}\label{lem:forms}
For every $s\in(0,1)$ and every $u\in H^1_0(\Ocal)$,
\begin{equation}\label{eq:form-comparison}
\mathcal M_s(u,u)=[u]^2_{H^s(\R^d)}\;\le\;C_s\,\norm{\nabla u}^2,\qquad
C_s:=1+\lambda_1^{-1},
\end{equation}
with $C_s$ \emph{independent of $s$}. Consequently $\normV{\cdot}$ is
equivalent to $\norm{\nabla\cdot}$,
\begin{equation}\label{eq:V-equiv}
\theta\norm{\nabla u}^2\;\le\;\normV{u}^2\;\le\;(\theta+\beta C_s)
\norm{\nabla u}^2,
\end{equation}
so $V=H^1_0(\Ocal)$ and $V^*=H^{-1}(\Ocal)$, and
\[
V\hookrightarrow H\hookrightarrow V^*,\qquad
H^1_0(\Ocal)\hookrightarrow L^2(\Ocal)\hookrightarrow H^{-1}(\Ocal),
\]
is a Gelfand triple with dense, compact embeddings and pairing
$[\cdot,\cdot]_{V^*,V}$ extending $\ip{\cdot}{\cdot}$.
\end{lemma}

\begin{proof}
With the standard normalisation of $C(d,s)$,
$[u]^2_{H^s(\R^d)}=\int_{\R^d}|\xi|^{2s}|\hat u(\xi)|^2\,d\xi$
\cite[Prop.~3.6]{DiNezza-Palatucci2012}; since $|\xi|^{2s}\le1+|\xi|^2$ for
all $\xi$ and all $s\in(0,1)$, Plancherel gives
$[u]^2_{H^s}\le\norm u^2+\norm{\nabla u}^2$, and Poincar\'e
$\norm u^2\le\lambda_1^{-1}\norm{\nabla u}^2$ yields
\eqref{eq:form-comparison}. Then \eqref{eq:V-equiv} is immediate, and
$H^1_0\hookrightarrow\widetilde H^s$ for every $s\in(0,1)$ identifies the
space.
\end{proof}

Note that $\abs{u}^pu\in L^{(p+2)'}(\Ocal)$ with
$\normLp{\abs{u}^pu}{(p+2)'}=\normLp{u}{p+2}^{p+1}$. Then every term in the
weak formulation \eqref{eq:weak2} below is finite for
$\varphi\in V\hookrightarrow L^{p+2}$; the noise series converges in
$L^2(\Omega;H)$ by Assumption~\ref{ass:standing}(i).

\begin{remark}[Comparison with competitive "wrong sign" operators] \label{rem:wrong_sign}
It is instructive to contrast the cooperative structure of our mixed operator $\Acal = -\theta\Delta + \beta(-\Delta)^s$ (often referred to as having the ``right sign'' \cite{BiagiDipierroValdinociVecchi2021}) with the non-cooperative or competitive regimes recently studied in the literature. For instance, in the recent preprint \cite{BiagiDipierroValdinociVecchi2026Wrong}, the authors investigate Sobolev critical elliptic problems involving the operator $-\Delta - \gamma(-\Delta)^s$, where the nonlocal fractional Laplacian appears with the ``wrong sign'' and acts as a competitive perturbation of the classical Laplacian. 

In such non-cooperative models, the operator is positive definite if and only if the competitive parameter $\gamma$ is strictly bounded from above by the optimal fractional Sobolev embedding constant:
\begin{equation} \label{eq:Cemb_def}
C_{\mathrm{emb}} := \inf \left\{ \norm{\nabla u}^2 : u \in H^1_0(\Ocal), \ [u]^2_{H^s(\R^d)} = 1 \right\} > 0.
\end{equation}
When $\gamma \in (0, C_{\mathrm{emb}})$, coercivity is recovered via the estimate $\norm{\nabla u}^2 - \gamma [u]^2_{H^s(\R^d)} \ge (1 - \gamma/C_{\mathrm{emb}})\norm{\nabla u}^2$ \cite{BiagiDipierroValdinociVecchi2026Wrong}. 

In stark contrast, in our setting, both coefficients $\theta$ and $\beta$ are strictly positive. Since the local and nonlocal energy components act cooperatively, the positive definiteness of $\Acal$ and the norm equivalence \eqref{eq:V-equiv} established in Lemma~\ref{lem:forms} hold unconditionally for any choice of positive parameters $\theta, \beta > 0$. This structural stability of the ``right sign'' operator is mathematically fundamental to our work, as it guarantees that the energy functional controls the phase-space norm globally in time without requiring any spectral threshold restrictions on the coefficients.
\end{remark}

\section{The linear stochastic group and the stochastic convolution}
\label{sec:convolution}

Throughout, $\Acal:=-\theta\Delta+\beta(-\Delta)^s$ is the mixed operator,
$H:=L^2(\Ocal)$, $V:=D(\Acal^{1/2})=H^1_0(\Ocal)$ with
$\normV{u}^2=\ip{\Acal u}{u}$ (Lemma~\ref{lem:forms}), and
\[
  \Vcal:=V\times H,\qquad
  \norm{(u,v)}_{\Vcal}^2:=\normV{u}^2+\normH{v}^2 .
\]
By Lemma~\ref{lem:forms} the operator $\Acal$ is self-adjoint, positive and
boundedly invertible on $H$ (indeed $\Acal\ge\theta\lambda_1>0$), with
compact resolvent; hence
\begin{equation}\label{eq:group-def}
  S(t)=\begin{pmatrix}\cos(t\Acal^{1/2}) & \Acal^{-1/2}\sin(t\Acal^{1/2})\\[2pt]
       -\Acal^{1/2}\sin(t\Acal^{1/2}) & \cos(t\Acal^{1/2})\end{pmatrix},
  \qquad t\in\R,
\end{equation}
is a \emph{unitary group} on $\Vcal$: $S(t)S(r)=S(t+r)$, $S(0)=\mathrm{Id}$
and $\norm{S(t)}_{\Lcal(\Vcal)}=1$ for every $t\in\R$. We write
$\Gbb:H\to\Vcal$, $\Gbb g:=(0,g)^{\!\top}$, so that
$\norm{\Gbb g}_{\Vcal}=\normH{g}$, and
\begin{equation}\label{eq:WQ}
  W_Q(t):=\sigma\sum_{k\ge1}\lambda_k^{-\gamma}h_k\,w_k(t),
  \qquad Q:=\sigma^2\sum_{k\ge1}\lambda_k^{-2\gamma}h_k\otimes h_k,
  \qquad \Tr Q=\kappa .
\end{equation}

\begin{definition}[Stochastic convolution]\label{def:z}
The \emph{stochastic convolution} associated with \eqref{eq:group-def} is the
$\Vcal$-valued process
\begin{equation}\label{eq:Zdef}
  Z(t):=\bigl(z(t),\,\partial_t z(t)\bigr)^{\!\top}
  :=\int_0^t S(t-r)\,\Gbb\,dW_Q(r),
  \qquad t\ge0 ;
\end{equation}
equivalently, $z$ is the unique solution of the \emph{linear} problem
$\partial_t^2z+\Acal z=\sigma\dot W_Q$ on $\Ocal$, $z=0$ on
$\R^d\setminus\Ocal$, $z(0)=\partial_tz(0)=0$.
\end{definition}

\begin{lemma}\label{lem:z}
Let Assumption~\ref{ass:standing}\emph{(i)} hold, i.e.\ $\gamma>d/4$ and
$\kappa=\Tr Q<\infty$. Then:
\begin{enumerate}
\item[\emph{(i)}] $Z$ is a well-defined, $\{\mathcal F_t\}$-adapted, centred
Gaussian random element of $C([0,T];\Vcal)$ for every $T>0$, and
\begin{equation}\label{eq:Zmoments}
  \E\sup_{t\le T}\norm{Z(t)}_{\Vcal}^{q}\;\le\;C_q\,(\kappa T)^{q/2},
  \qquad q\in[2,\infty),
\end{equation}
and moreover
$\E\exp\bigl(\varepsilon\sup_{t\le T}\norm{Z(t)}^2_{\Vcal}\bigr)<\infty$
for some $\varepsilon=\varepsilon(\kappa,T)>0$.
\item[\emph{(ii)}] For every $\varepsilon>0$,
$\ \PP\bigl(\sup_{t\le T}\norm{Z(t)}_{\Vcal}\le\varepsilon\bigr)>0 .$
\item[\emph{(iii)}] Let $\delta:=d\bigl(\tfrac12-\tfrac1{p+2}\bigr)$, which
satisfies $\delta\in(0,1)$ under Assumption~\ref{ass:standing}\emph{(ii)}. If
in addition \eqref{eq:gamma-strong} holds, then $Z$ has a version with
$z\in C([0,T];D(\Acal^{(1+\delta)/2}))$ and
$\partial_tz\in C\bigl([0,T];D(\Acal^{\delta/2})\bigr)\hookrightarrow
C\bigl([0,T];L^{p+2}(\Ocal)\bigr)$,
and \eqref{eq:Zmoments}, \emph{(ii)} hold with $\norm{\cdot}_{\Vcal}$
replaced by $\normV{z}+\normLp{\partial_tz}{p+2}$.
\end{enumerate}
\end{lemma}

\begin{proof}
(i) Since $S$ is a group, $S(t-r)=S(t)S(-r)$, whence
\begin{equation}\label{eq:group-trick}
  Z(t)=S(t)\,Y(t),\qquad Y(t):=\int_0^t S(-r)\,\Gbb\,dW_Q(r).
\end{equation}
For every $r$,
$\norm{S(-r)\Gbb Q^{1/2}}^2_{L_2(H,\Vcal)}
=\sum_k\norm{S(-r)\Gbb Q^{1/2}h_k}^2_{\Vcal}
=\sum_k\normH{Q^{1/2}h_k}^2=\Tr Q=\kappa$,
because $S(-r)$ is a $\Vcal$-isometry and $\norm{\Gbb g}_{\Vcal}=\normH g$.
Hence the integrand is deterministic and Hilbert--Schmidt uniformly in $r$,
so $Y$ is a continuous square-integrable $\Vcal$-valued martingale with
$\E\norm{Y(t)}^2_{\Vcal}=\kappa t$, and Burkholder--Davis--Gundy in Hilbert
space \cite{DaPratoZabczyk2014} gives
$\E\sup_{t\le T}\norm{Y(t)}^q_{\Vcal}\le C_q(\kappa T)^{q/2}$. Since $S(t)$
is an isometry of $\Vcal$ and $(t,y)\mapsto S(t)y$ is jointly continuous,
\eqref{eq:group-trick} yields $Z\in C([0,T];\Vcal)$ with
$\sup_{t\le T}\norm{Z}_{\Vcal}=\sup_{t\le T}\norm{Y}_{\Vcal}$, and
\eqref{eq:Zmoments}. $Z$ is a centred Gaussian element of the separable
Banach space $C([0,T];\Vcal)$ (being an $L^2$-limit of Gaussian Riemann
sums), so Fernique's theorem \cite{DaPratoZabczyk2014} gives the exponential
moment.

(ii) The support of a centred Gaussian measure on a separable Banach space is
a \emph{closed linear subspace} (the closure of its Cameron--Martin space);
in particular it contains the origin, so every ball centred at $0$ has
positive mass.

(iii) Put $\Vcal_\rho:=D(\Acal^{(1+\rho)/2})\times D(\Acal^{\rho/2})$ for
$\rho\in[0,1]$. Because $S$ commutes with every power of $\Acal$, $S$ is also
a unitary group on $\Vcal_\rho$, and repeating the computation in (i) in
$\Vcal_\rho$ gives
\[
  \norm{S(-r)\Gbb Q^{1/2}}^2_{L_2(H,\Vcal_\rho)}
   =\sigma^2\sum_{k\ge1}\lambda_k^{-2\gamma}\,\norm{\Acal^{\rho/2}h_k}^2 .
\]
By Lemma~\ref{lem:forms},
$\theta(-\Delta)\le\Acal\le(\theta+\beta C_s)(-\Delta)$ in the sense of
quadratic forms on $H^1_0(\Ocal)$; since $\rho/2\le1/2$, the Heinz--Kato
inequality upgrades this to the norm comparison
\begin{equation}\label{eq:heinz}
  \theta^{\rho/2}\norm{(-\Delta)^{\rho/2}u}\;\le\;\norm{\Acal^{\rho/2}u}
  \;\le\;(\theta+\beta C_s)^{\rho/2}\norm{(-\Delta)^{\rho/2}u},
  \qquad \rho\in[0,1],
\end{equation}
so $\norm{\Acal^{\rho/2}h_k}\le(\theta+\beta C_s)^{\rho/2}\lambda_k^{\rho/2}$
and $D(\Acal^{\rho/2})=H^\rho_0(\Ocal)$ with equivalent norms. Consequently
\[
 \norm{S(-r)\Gbb Q^{1/2}}^2_{L_2(H,\Vcal_\rho)}
 \;\lesssim\;\sum_{k\ge1}\lambda_k^{\rho-2\gamma}
 \;\asymp\;\sum_{k\ge1}k^{\frac2d(\rho-2\gamma)},
\]
which converges iff $\tfrac2d(2\gamma-\rho)>1$, i.e.\ iff
$\gamma>\tfrac d4+\tfrac\rho2$. Taking $\rho=\delta$ this is exactly
\eqref{eq:gamma-strong}, and
$H^\delta_0(\Ocal)\hookrightarrow L^{p+2}(\Ocal)$ by the Sobolev embedding
with $\tfrac1{p+2}=\tfrac12-\tfrac\delta d$. That $\delta<1$ follows from
$p+2<2^{*}=\tfrac{2d}{d-2}$. The remaining assertions are (i)--(ii) applied
in $\Vcal_\delta$.
\end{proof}

\section{Galerkin approximation for the mixed local–nonlocal operator}\label{sec:galerkin}

\subsection{Non-diagonal fractional stiffness matrix}

Let $H_N=\operatorname{span}\{h_1,\dots,h_N\}$ with $L^2$-orthogonal
projection $\Pi^N$,
$\Lambda^{(N)}=\operatorname{diag}(\lambda_1,\dots,\lambda_N)$, and let
$M^{(N)}$ be the projection of the fractional form onto $H_N$,
\begin{equation}\label{eq:Mmatrix}
M^{(N)}_{kj}:=\mathcal M_s(h_k,h_j)=\frac{C(d,s)}{2}
\iint_{\R^d\times\R^d}
\frac{(h_k(x)-h_k(y))(h_j(x)-h_j(y))}{|x-y|^{d+2s}}\,dx\,dy .
\end{equation}

\begin{lemma}\label{lem:Mproperties}
For every $N\in\N$, $M^{(N)}$ is symmetric and positive semidefinite, and for
every $U=(u_1,\dots,u_N)\in\R^N$, writing $u=\sum_{k\le N}u_kh_k$,
\begin{equation}\label{eq:Mbounds}
0\;\le\;\ip{M^{(N)}U}{U}_{\R^N}=\mathcal M_s(u,u)=[u]^2_{H^s(\R^d)}
\;\le\;C_s\norm{\nabla u}^2=C_s\sum_{k\le N}\lambda_k u_k^2 .
\end{equation}
Consequently $A^{(N)}:=\theta\Lambda^{(N)}+\beta M^{(N)}$ is symmetric with
\begin{equation}\label{eq:A-bounds}
\theta\lambda_1\abs{U}^2\;\le\;\theta\ip{\Lambda^{(N)}U}{U}
\;\le\;\ip{A^{(N)}U}{U}
\;\le\;(\theta+\beta C_s)\ip{\Lambda^{(N)}U}{U},\qquad U\in\R^N.
\end{equation}
\end{lemma}

\begin{proof}
Symmetry is manifest in \eqref{eq:Mmatrix}. By bilinearity
$\ip{M^{(N)}U}{U}=\mathcal M_s(u,u)=[u]^2_{H^s(\R^d)}\ge0$; the upper bound
is Lemma~\ref{lem:forms} together with
$\norm{\nabla u}^2=\sum_k\lambda_k u_k^2$ on $H_N$. Inequality
\eqref{eq:A-bounds} follows since
$\theta\lambda_1|U|^2\le\theta\ip{\Lambda^{(N)}U}{U}$ and
$\beta\ip{M^{(N)}U}{U}\le\beta C_s\ip{\Lambda^{(N)}U}{U}$.
\end{proof}

\subsection{Uniform energy estimates}

We rewrite the defocusing SPDE ($\varepsilon=+1$ in
\eqref{stochastic_local_nonlocal}) as the first-order system
\begin{align}\label{stochastic_problem}
\begin{cases}
\displaystyle du =  v\, dt,&\\
\displaystyle dv =\theta \Delta u\, dt  -\beta (-\Delta)^{s}u\, dt
 -|u|^pu\, dt + \sigma \sum_{k\in\N} \lambda_k^{-\gamma}\, h_k(x)\,
 dw_k(t), &\text{in } (0,+\infty)\times\Ocal,\\
u = 0, &\text{on } (0,+\infty)\times (\R^d \setminus \Ocal),\\
u(0,x)= U_0(x),\quad v(0,x)= V_0(x), &\text{in } \Ocal.
\end{cases}
\end{align}
We approximate \eqref{stochastic_problem} by projecting onto $H_N$ and
writing the first-order (wave) system for the pair $(U^N,V^N)$ with
$u^N=\sum_{k\le N}u^N_kh_k$, $v^N=\sum_{k\le N}v^N_kh_k$. With the projected
nonlinearity
\begin{equation}\label{eq:Fdef}
F^{(N)}_k(U):=\int_{\Ocal}\Bigl|\sum_{j\le N}u_jh_j\Bigr|^{p}
\Bigl(\sum_{j\le N}u_jh_j\Bigr)h_k\,dx,\qquad k=1,\dots,N,
\end{equation}
the system reads
\begin{equation}\label{eq:galerkin-system}
\begin{cases}
du^N_k=v^N_k\,dt, & k=1,\dots,N,\\[4pt]
\displaystyle
dv^N_k=-\theta\lambda_k u^N_k\,dt-\beta\sum_{j=1}^N M^{(N)}_{kj}u^N_j\,dt
-F^{(N)}_k(U^N)\,dt+\sigma\lambda_k^{-\gamma}\,dw_k(t), & k=1,\dots,N,\\[6pt]
u^N_k(0)=\ip{U_0}{h_k},\qquad v^N_k(0)=\ip{V_0}{h_k}.
\end{cases}
\end{equation}
Its energy and the coercivity used repeatedly below are
\begin{equation}\label{eq:energyN}
\Ecal^N(t):=\tfrac12\abs{V^N(t)}^2
+\tfrac12\ip{A^{(N)}U^N(t)}{U^N(t)}
+\tfrac{1}{p+2}\normLp{u^N(t)}{p+2}^{p+2},
\end{equation}
\begin{equation}\label{eq:coercive}
\Ecal^N(U,V)\;\ge\;\tfrac12\abs V^2+\tfrac{\theta\lambda_1}{2}\abs U^2
\;\ge\;c_0\abs X^2,\qquad
c_0:=\tfrac12\min\{1,\theta\lambda_1\},\quad X=(U,V),
\end{equation}
the latter by \eqref{eq:A-bounds} and
$\Phi(U):=\frac1{p+2}\normLp{u}{p+2}^{p+2}\ge0$.

\subsection{Energy solutions}

Because the noise is \emph{additive}, the stochastic term below is an
explicit $H$-valued Wiener process and no It\^o integral appears in the
definition.

\begin{definition}[Energy solution]\label{def:solution}
Let $T>0$. An \emph{energy solution} of \eqref{stochastic_problem} on
$[0,T]$ is a pair $(u,v)$ of $\{\mathcal F_t\}$-adapted processes such that,
$\PP$-a.s.,
\[
u\in C([0,T];V),\qquad v\in C([0,T];H),
\]
\begin{equation}\label{eq:sol-class}
\E\sup_{t\le T}\Bigl(\normV{u(t)}^2+\norm{v(t)}^2
+\normLp{u(t)}{p+2}^{p+2}\Bigr)<\infty ,
\end{equation}
and for every $\varphi\in V$ and every $t\in[0,T]$, $\PP$-a.s.,
\begin{align}
\ip{u(t)}{\varphi}&=\ip{U_0}{\varphi}+\int_0^t\ip{v(r)}{\varphi}\,dr,
\label{eq:weak1}\\[2pt]
\ip{v(t)}{\varphi}&=\ip{V_0}{\varphi}
-\int_0^t\Bigl[\theta\ip{\nabla u(r)}{\nabla\varphi}
+\beta\,\mathcal M_s\bigl(u(r),\varphi\bigr)
+\ip{\abs{u(r)}^pu(r)}{\varphi}\Bigr]dr \notag\\
&\qquad+\;\sigma\sum_{k\ge1}\lambda_k^{-\gamma}\ip{h_k}{\varphi}\,w_k(t).
\label{eq:weak2}
\end{align}
The solution is \emph{global} if it is an energy solution on $[0,T]$ for
every $T>0$.
\end{definition}

\subsection{Auxiliary lemmas}

We first prove that for each $N\in\N$ the system
\eqref{eq:galerkin-system} has a unique global strong solution
$X^N=(U^N,V^N)\in C([0,\infty);\R^{2N})$ a.s., with an energy bound uniform
in $N$.

\begin{lemma}\label{lem:Fproperties}
Let $U=(u_1,\dots,u_N)\in\R^N$, $u=\sum_{k\le N}u_kh_k$ and
$\Phi(U):=\frac{1}{p+2}\normLp{u}{p+2}^{p+2}$. Then:
\begin{enumerate}
\item[\emph{(a)}] $\Phi\in C^2(\R^N)$, with $\nabla\Phi=F^{(N)}$ given by
\eqref{eq:Fdef} and
$D^2\Phi(U)_{kj}=(p+1)\int_\Ocal\abs{u}^p h_kh_j\,dx$;
\item[\emph{(b)}] $F^{(N)}$ is locally Lipschitz on $\R^N$ and
$\abs{F^{(N)}(U)}\le C_N\bigl(1+\abs U^{p+1}\bigr)$;
\item[\emph{(c)}]
$\ip{F^{(N)}(U)}{U}_{\R^N}=\normLp{u}{p+2}^{p+2}=(p+2)\Phi(U)\ge0$.
\end{enumerate}
\end{lemma}

\begin{proof}
Since $H_N\subset H^1_0(\Ocal)\hookrightarrow L^{p+2}(\Ocal)$ is finite
dimensional, all norms on $H_N$ are equivalent; in particular there is
$c_N>0$ with
\begin{equation}\label{eq:norm-equiv-HN}
\normLp{u}{p+2}\le c_N\abs U,\qquad U\in\R^N .
\end{equation}
(a) The map $r\mapsto\abs r^pr$ is $C^1$ on $\R$ with derivative
$(p+1)\abs r^p$ for every $p>0$; hence
$r\mapsto\frac{1}{p+2}\abs r^{p+2}$ is $C^2$. Differentiating twice under
the integral sign --- justified on bounded sets of $\R^N$ by
\eqref{eq:norm-equiv-HN} and dominated convergence --- gives (a).
(b) From
$\bigl|\abs a^pa-\abs b^pb\bigr|\le(p+1)(\abs a+\abs b)^p\abs{a-b}$; for the
growth inequality apply H\"older's inequality with exponents
$\frac{p+2}{p},\,p+2,\,p+2$ and \eqref{eq:norm-equiv-HN}.
(c) Immediate from the definition. For $N=1$,
$F^{(1)}(u_1)=\abs{u_1}^pu_1\normLp{h_1}{p+2}^{p+2}$.
\end{proof}

\begin{lemma}\label{lem:initial-energy}
Let Assumption~\ref{ass:standing} hold and let
$\Ecal^N(0):=\frac12\abs{\Pi^NV_0}^2
+\frac12\ip{A^{(N)}\Pi^NU_0}{\Pi^NU_0}_{\R^N}
+\frac{1}{p+2}\normLp{\Pi^NU_0}{p+2}^{p+2}$. Then
\begin{equation}\label{eq:E0star}
\Ecal^{*}_0:=\sup_{N\ge1}\Ecal^N(0)\;<\;\infty,\qquad
\Ecal^{*}_0\le\tfrac12\normLp{V_0}{2}^2
+\tfrac{\theta+\beta C_s}{2}\normLp{\nabla U_0}{2}^2
+\tfrac{1}{p+2}\sup_N\normLp{\Pi^NU_0}{p+2}^{p+2},
\end{equation}
and moreover $\Ecal^N(0)\to\Ecal(0)$ as $N\to\infty$, where
$\Ecal(0):=\frac12\normLp{V_0}{2}^2+\frac{\theta}{2}\normLp{\nabla
U_0}{2}^2+\frac{\beta}{2}[U_0]^2_{H^s}+\frac{1}{p+2}\normLp{U_0}{p+2}^{p+2}$.
\end{lemma}

\begin{proof}
$\abs{\Pi^NV_0}\le\normLp{V_0}{2}$ by Bessel. For the quadratic part,
$\normLp{\nabla\Pi^NU_0}{2}^2
=\sum_{k\le N}\lambda_k\ip{U_0}{h_k}_{L^2}^2\le\normLp{\nabla U_0}{2}^2$,
and therefore, by \eqref{eq:A-bounds} applied to $U=\Pi^NU_0$,
\[
\ip{A^{(N)}\Pi^NU_0}{\Pi^NU_0}\le(\theta+\beta C_s)
\norm{\nabla\Pi^NU_0}^2\le(\theta+\beta C_s)\norm{\nabla U_0}^2 .
\]
For the $L^{p+2}$ part, $\Pi^NU_0\to U_0$ in $H^1_0(\Ocal)$ because
$\normLp{\nabla(U_0-\Pi^NU_0)}{2}^2
=\sum_{k>N}\lambda_k\ip{U_0}{h_k}^2\to0$; by
Assumption~\ref{ass:standing}(ii) and the Sobolev embedding,
$\Pi^NU_0\to U_0$ in $L^{p+2}(\Ocal)$, so the last supremum is finite.
Finally
$\ip{A^{(N)}\Pi^NU_0}{\Pi^NU_0}_{\R^N}
=\theta\normLp{\nabla\Pi^NU_0}{2}^2+\beta[\Pi^NU_0]^2_{H^s}
\to\theta\normLp{\nabla U_0}{2}^2+\beta[U_0]^2_{H^s}$ by continuity of the
Gagliardo seminorm on $H^1_0$ (Lemma~\ref{lem:forms}), whence
$\Ecal^N(0)\to\Ecal(0)$.
\end{proof}

\begin{lemma}\label{lem:galerkin}
Let Assumption~\ref{ass:standing} hold. For each $N$ the system
\eqref{eq:galerkin-system} has a unique global strong solution
$X^N=(U^N,V^N)\in C([0,\infty);\R^{2N})$ a.s. Moreover, with $\Ecal^N$
defined by \eqref{eq:energyN}, for every $T>0$,
\begin{equation}\label{eq:uniform-bound}
\E\sup_{t\in[0,T]}\Ecal^N(t)\;\le\;2\Ecal^{*}_0+C T,
\end{equation}
with a right-hand side independent of $N$.
\end{lemma}

\begin{proof}
Write the system as $dX=b(X)\,dt+\Sigma\,dW^{(N)}$ with
$b(U,V)=\bigl(V,\,-A^{(N)}U-F^{(N)}(U)\bigr)$ and
$\Sigma=\bigl(0,\,\sigma(\Lambda^{(N)})^{-\gamma}\bigr)^{\!\top}$.

\smallskip\noindent
$b$ is locally Lipschitz on $\R^{2N}$ by Lemma~\ref{lem:Fproperties}(b) and
$\Sigma$ is constant, so there is a unique strong solution up to
$\tau^N_\infty=\lim_{R\to\infty}\tau_R$,
$\tau_R:=\inf\{t\ge0:\abs{X^N(t)}\ge R\}$, with
$\abs{X^N(t)}\to\infty$ as $t\uparrow\tau^N_\infty$ on
$\{\tau^N_\infty<\infty\}$.

\smallskip\noindent
By Lemma~\ref{lem:Fproperties}(a), $\Ecal^N\in C^2(\R^{2N})$ with
$\nabla_U\Ecal^N=A^{(N)}U+F^{(N)}(U)$, $\nabla_V\Ecal^N=V$ and
$D^2_V\Ecal^N=\mathrm{Id}$, so It\^o's formula applies and the generator
$\Lcal$ of \eqref{eq:galerkin-system} applied to $\Ecal^N$ gives the exact
cancellation characteristic of Hamiltonian systems:
\begin{align}
\Lcal\,\Ecal^N(U,V)
&=\ip{\nabla_U\Ecal^N}{V}_{\R^N}
+\ip{\nabla_V\Ecal^N}{-A^{(N)}U-F^{(N)}(U)}_{\R^N}
+\tfrac12\Tr\bigl(\Sigma\Sigma^{\!\top}D^2\Ecal^N\bigr)\notag\\
&=\ip{A^{(N)}U+F^{(N)}(U)}{V}_{\R^N}
-\ip{V}{A^{(N)}U+F^{(N)}(U)}_{\R^N}
+\tfrac{\sigma^2}{2}\sum_{k\le N}\lambda_k^{-2\gamma}
\;=\;\tfrac{\kappa_N}{2},\label{eq:generator}
\end{align}
i.e.\ $\Lcal\Ecal^N\equiv\kappa_N/2$ with
$\kappa_N=\sigma^2\sum_{k\le N}\lambda_k^{-2\gamma}\le\kappa$. Moreover
$\Ecal^N$ is coercive: by \eqref{eq:coercive},
$q_R:=\inf_{\abs X\ge R}\Ecal^N(X)\ge c_0R^2$ with
$c_0=\frac12\min\{1,\theta\lambda_1\}$. By It\^o's formula on
$[0,t\wedge\tau_R]$,
\begin{equation}\label{eq:ito-energy}
\Ecal^N(t\wedge\tau_R)=\Ecal^N(0)+\tfrac{\kappa_N}{2}\,(t\wedge\tau_R)
+\int_0^{t\wedge\tau_R}\!\!\ip{V^N(r)}
{\sigma(\Lambda^{(N)})^{-\gamma}dW^{(N)}(r)} ,
\end{equation}
and since the last stochastic integral is a martingale on
$[0,t\wedge\tau_R]$,
\[
\E\,\Ecal^N(t\wedge\tau_R)=\Ecal^N(0)
+\tfrac{\kappa_N}{2}\,\E(t\wedge\tau_R)
\;\le\;\Ecal^{*}_0+\tfrac{\kappa}{2}t.
\]
Note $|X^N (\tau_R)|=R$ on $\{\tau_R\le t \}$, and there
$t\wedge \tau_R=\tau_R$, hence
$\Ecal^N( t \wedge \tau_R)=\Ecal^N (X^N (\tau_R))\ge q_R$. Then Chebyshev
gives
$\PP(\tau_R\le t)\le \PP(\Ecal^N(t\wedge\tau_R)\ge q_R)
\le \frac{\E\,\Ecal^N(t\wedge\tau_R)}{q_R}
\le(\Ecal^{*}_0+\frac\kappa2t)/(c_0R^2)\to0$ as $R\to \infty$, and since
$\{\tau^N_\infty\le t\}\subset\{\tau_R\le t\}$ for every $R$, we get
$\PP(\tau^N_\infty\le t)=0$ for all $t$, i.e.\ $\tau^N_\infty=\infty$ a.s.\
(Khasminskii's test, \cite{Khasmiskii2012,CHOW-Khasminski2014}).

\smallskip\noindent
Let $m(t):=\int_0^t\ip{V^N}{\sigma(\Lambda^{(N)})^{-\gamma}dW^{(N)}}$, so
that $\Ecal^N(t)=\Ecal^N(0)+\frac{\kappa_N}{2}t+m(t)$ and, by
\eqref{eq:coercive},
\[
\langle m\rangle_t
=\sigma^2\!\int_0^t\!\sum_{k\le N}\lambda_k^{-2\gamma}\abs{v^N_k}^2dr
\;\le\;\sigma^2\lambda_1^{-2\gamma}\!\int_0^t\!\abs{V^N}^2dr
\;\le\;C_1\!\int_0^t\!\Ecal^N(r)\,dr,
\qquad C_1:=2\sigma^2\lambda_1^{-2\gamma}.
\]
Fix $R>0$ and put $\Psi_R(T):=\E\sup_{t\le T\wedge\tau_R}\Ecal^N(t)$, which
is \emph{finite} because $\Ecal^N$ is continuous and hence bounded on
$\{\abs X\le R\}$. Taking suprema and expectations in \eqref{eq:ito-energy}
we have
\begin{equation}\label{eq:psi_ineq}
    \Psi_R(T)\le \Ecal_0^{*}+\frac{\kappa T}{2}
    +\E \sup_{t\le T\wedge\tau_R} |m(t)| .
\end{equation}
By the Burkholder--Davis--Gundy inequality,
\[
\E\sup_{t\le T\wedge\tau_R}\abs{m(t)}
\;\le\;C\,\E\,\langle m\rangle_{T\wedge\tau_R}^{1/2}
\le C\,\E\Bigl(C_1\int_0^{T\wedge \tau_R} \Ecal^N dr\Bigr)^{1/2}
\le C\sqrt{C_1 T}\;\E\bigl[S^{1/2}\bigr],
\]
where $S:=\sup_{t\le T\wedge\tau_R} \Ecal^N$. By Jensen
$\E[S^{1/2}]\le (\Psi_R (T))^{1/2}$, and Young's inequality gives
$\E\sup_{t\le T\wedge\tau_R}\abs{m(t)}
\le \frac{1}{2}\Psi_R (T)+\frac{C^2 C_1 T}{2}$. Combining with
\eqref{eq:psi_ineq} and absorbing $\frac12\Psi_R(T)$ into the left-hand
side, we obtain
\[
\Psi_R (T)=\E\sup_{t\le T\wedge\tau_R}\Ecal^N(t)
\le 2\Ecal_0^{*}+(\kappa+C^2C_1) T .
\]
Letting $R\to\infty$ (monotone convergence, $\tau_R\uparrow\infty$) yields
\eqref{eq:uniform-bound} uniformly in $N$.
\end{proof}

\section{The defocusing case: global well-posedness}\label{sec:defocusing}

\subsection{Pathwise reduction}

\begin{lemma}\label{lem:reduction}
Let $z$ be as in Definition~\ref{def:z} and set $y:=u-z$. Then
$(u,\partial_tu)$ is an energy solution of the defocusing problem
\eqref{stochastic_problem} on $[0,T]$ in the sense of
Definition~\ref{def:solution} if and only if, $\PP$-a.s.\ and for
$\PP$-a.e.\ fixed $\omega$, $y$ solves the \emph{deterministic} problem
\begin{equation}\label{eq:reduced-defoc}
  \partial_t^2y+\Acal y=-\,\abs{y+z}^{p}(y+z),\qquad
  y(0)=U_0,\quad \partial_ty(0)=V_0,
\end{equation}
in $C([0,T];V)\cap C^1([0,T];H)$, i.e.\ in the Duhamel sense
\begin{equation}\label{eq:duhamel-defoc}
  \bigl(y,\partial_ty\bigr)(t)=S(t)(U_0,V_0)
  -\int_0^t S(t-r)\bigl(0,\,\abs{y+z}^p(y+z)(r)\bigr)\,dr .
\end{equation}
\end{lemma}

\begin{proof}
Because the noise is additive and $z$ solves the linear problem with zero
data, $u=y+z$ satisfies \eqref{eq:weak1}--\eqref{eq:weak2} iff $y$ satisfies
the same identities with the stochastic term deleted and the nonlinearity
evaluated at $y+z$; this is \eqref{eq:reduced-defoc}. The equivalence of the
weak and Duhamel formulations for the group $S$ is classical. Note that no
It\^o integral survives: \eqref{eq:reduced-defoc} is an ODE in $\Vcal$ for
each fixed $\omega$.
\end{proof}

\begin{theorem}[Global well-posedness, defocusing case]\label{thm:global-dd}
Let Assumption~\ref{ass:standing} hold, let $p$ satisfy the
energy-subcritical range \eqref{eq:local-range}, and let $\gamma$ satisfy
\eqref{eq:gamma-strong}. Let $U_0\in V$, $V_0\in H$ be
$\mathcal F_0$-measurable. Then \eqref{stochastic_problem} has a unique
global energy solution; it is probabilistically strong,
$\{\mathcal F_t\}$-adapted, and
\begin{equation}\label{eq:global-class}
  (u,\partial_tu)\in C\bigl([0,\infty);\Vcal\bigr)\quad\PP\text{-a.s.}
\end{equation}
Moreover, with
$e(t):=\tfrac12\normH{\partial_ty(t)}^2+\tfrac12\normV{y(t)}^2
+\tfrac1{p+2}\normLp{u(t)}{p+2}^{p+2}$ and
$\psi(t):=\normLp{\partial_tz(t)}{p+2}$ one has the \emph{pathwise} bound
\begin{equation}\label{eq:pathwise-gronwall}
  1+e(t)\;\le\;\bigl(1+e(0)\bigr)
  \exp\Bigl(C_\star\!\int_0^t\!\psi(r)\,dr\Bigr),
  \qquad t\ge0,\quad \PP\text{-a.s.},
\end{equation}
with $C_\star=C_\star(p)$; consequently, if
$\E\bigl[(1+e(0))^{q}\bigr]<\infty$ for some $q\ge1$ then
$\E\sup_{t\le T}\bigl(\normH{\partial_tu}^2+\normV{u}^2
+\normLp{u}{p+2}^{p+2}\bigr)^{q}<\infty$ for every $T>0$.
\end{theorem}

\begin{proof}
\emph{Step 1 (local solvability).} We will prove local existence using contraction argument. Fix $\omega$ outside the null
set of Lemma~\ref{lem:z}(i). By \eqref{eq:local-range},
$V\hookrightarrow L^{2(p+1)}(\Ocal)$ with constant $C_S$, and for
$w_1,w_2,\zeta\in V$
\begin{align}
  \norm{N(w_1,\zeta)-N(w_2,\zeta)}_{H}
  &\le(p+1)C_S^{p+1}\bigl(\normV{w_1+\zeta}
  +\normV{w_2+\zeta}\bigr)^{p}\normV{w_1-w_2},
  \label{eq:Nlip-defoc}\\
  \norm{N(w,\zeta)}_{H}&\le C_S^{p+1}\normV{w+\zeta}^{p+1},
  \label{eq:Ngrowth-defoc}
\end{align}
where $N(w,\zeta):=\abs{w+\zeta}^p(w+\zeta)$; both follow from
$\bigl||a|^pa-|b|^pb\bigr|\le(p+1)(|a|+|b|)^p|a-b|$ and H\"older on $L^2$
with exponents $\frac{p+1}{p}$ and $p+1$. 
Fix $T_0>0$ and put
\begin{equation}\label{eq:Rdef}
  R:=1+\norm{(U_0,V_0)}_{\Vcal}+\sup_{t\le T_0}\norm{Z(t)}_{\Vcal}
  \in[1,\infty),
\end{equation}
finite for the fixed $\omega$ by Lemma~\ref{lem:z}(i). For $T\in(0,T_0]$ set
\[
  X_{T,R}:=\Bigl\{Y=(y_1,y_2)\in C([0,T];\Vcal):\;
  \norm{Y}_T:=\sup_{t\le T}\norm{Y(t)}_{\Vcal}\le 2R\Bigr\},
\]
a closed subset of the Banach space $C([0,T];\Vcal)$ and hence a complete
metric space for the distance induced by $\norm\cdot_T$. Define the Duhamel
map
\begin{equation}\label{eq:duhamel-map}
  (\Tcal Y)(t):=S(t)(U_0,V_0)
  -\int_0^tS(t-r)\,\Gbb\,N\bigl(y_1(r),z(r)\bigr)\,dr,
  \qquad t\in[0,T],
\end{equation}
so that \eqref{eq:duhamel-defoc} is precisely the equation $\Tcal Y=Y$. The
map is well defined with values in $C([0,T];\Vcal)$ given by
\eqref{eq:Nlip-defoc} the map $N(\cdot,\zeta):V\to H$ is Lipschitz on
bounded sets, so $r\mapsto\Gbb N(y_1(r),z(r))$ is continuous from $[0,T]$
into $\Vcal$, and $t\mapsto S(t)$ is strongly continuous, whence the
integral term is continuous in $t$ by dominated convergence.
 
Two elementary facts are used repeatedly: $\norm{S(t)}_{\Lcal(\Vcal)}=1$ for
all $t$, and $\norm{\Gbb g}_{\Vcal}=\normH g$. Note also that
$\normV{z(r)}\le\norm{Z(r)}_{\Vcal}\le R$ for $r\le T_0$, so every
$Y\in X_{T,R}$ satisfies
\begin{equation}\label{eq:ball-bound}
  \normV{y_1(r)+z(r)}\le\normV{y_1(r)}+\normV{z(r)}\le 3R,
  \qquad r\le T .
\end{equation}
 
\emph{$\Tcal$ maps $X_{T,R}$ into itself.} By \eqref{eq:Ngrowth-defoc} and
\eqref{eq:ball-bound},
\[
  \norm{(\Tcal Y)(t)}_{\Vcal}
  \le\norm{(U_0,V_0)}_{\Vcal}+\int_0^t\normH{N(y_1,z)}\,dr
  \le R+T\,C_S^{p+1}(3R)^{p+1}.
\]
Hence $\norm{\Tcal Y}_T\le 2R$ as soon as
$T\,C_S^{p+1}3^{p+1}R^{p+1}\le R$, i.e.\ as soon as
\begin{equation}\label{eq:Tcond1}
  T\;\le\;(3C_S)^{-(p+1)}R^{-p} .
\end{equation}
 
\emph{$\Tcal$ is a $\tfrac12$-contraction.} For $Y,\widetilde Y\in X_{T,R}$
the initial-data terms in \eqref{eq:duhamel-map} cancel, so only the Duhamel
integral remains. Applying \eqref{eq:Nlip-defoc} with $\zeta=z(r)$,
$w_1=y_1(r)$, $w_2=\widetilde y_1(r)$, and bounding
$\bigl(\normV{y_1+z}+\normV{\widetilde y_1+z}\bigr)^p\le(6R)^p$ by
\eqref{eq:ball-bound},
\[
  \norm{(\Tcal Y)(t)-(\Tcal\widetilde Y)(t)}_{\Vcal}
  \le\int_0^t\normH{N(y_1,z)-N(\widetilde y_1,z)}\,dr
  \le T\,(p+1)C_S^{p+1}(6R)^p\,\norm{Y-\widetilde Y}_T ,
\]
where we used $\normV{y_1(r)-\widetilde y_1(r)}
\le\norm{Y(r)-\widetilde Y(r)}_{\Vcal}$. The contraction factor is therefore
at most $\tfrac12$ as soon as
\begin{equation}\label{eq:Tcond2}
  T\;\le\;\bigl[2(p+1)6^pC_S^{p+1}\bigr]^{-1}R^{-p} .
\end{equation}
Both \eqref{eq:Tcond1} and \eqref{eq:Tcond2} hold for
\begin{equation}\label{eq:Tlocal}
  T\;\le\;\underline T:=c_\star R^{-p}\wedge T_0,
  \qquad
  c_\star:=\min\Bigl\{(3C_S)^{-(p+1)},\,
  \bigl[2(p+1)6^pC_S^{p+1}\bigr]^{-1}\Bigr\},
\end{equation}
and $c_\star=c_\star(p,C_S)$. Banach's fixed point theorem gives a unique
$(y,\partial_ty)\in C([0,\underline T];\Vcal)$.

\emph{Step 2 (energy identity for the reduced problem).} On the interval of
existence, $\abs{u}^pu\in C([0,T];H)$ by \eqref{eq:Ngrowth-defoc}, so
$\partial_t^2y=-\Acal y-\abs u^pu\in C([0,T];V^{*})$ and
$\partial_ty\in C([0,T];H)$; testing \eqref{eq:reduced-defoc} with
$\partial_ty$ is admissible after the standard mollification in time
(\cite{LionsMagenes1972}; see also \cite[Ch.~III, Lemma~3.2]{Temam1997}),
and the chain rule for
$t\mapsto\frac{1}{p+2}\normLp{u(t)}{p+2}^{p+2}$ is justified because
$\abs{u}^{p+1}\in C([0,T];L^2)$ (again $V\hookrightarrow L^{2(p+1)}$) while
$\partial_tu\in C([0,T];H)$. We compute the three contributions separately, writing
$e=e_1+e_2+e_3$ with
\[
  e_1:=\tfrac12\normH{\partial_ty}^2,\qquad
  e_2:=\tfrac12\normV{y}^2=\tfrac12\Acal(y,y),\qquad
  e_3:=\tfrac1{p+2}\normLp{u}{p+2}^{p+2},
\]
where $\Acal(\cdot,\cdot)$ is the symmetric bilinear form \eqref{eq:frac-form}
associated with $\Acal$. Differentiating $e_1$ and
substituting the equation \eqref{eq:reduced-defoc},
$\partial_t^2y=-\Acal y-\abs u^pu$, gives
\begin{equation}\label{eq:e1prime}
  e_1'=\ip{\partial_t^2y}{\partial_ty}
      =-\ip{\Acal y}{\partial_ty}-\ip{\abs u^pu}{\partial_ty}.
\end{equation}
Because $\Acal(\cdot,\cdot)$
is bilinear and symmetric, the two cross terms
produced by the product rule are equal and the factor $\tfrac12$ is exactly
absorbed:
\begin{equation}\label{eq:e2prime}
  e_2'=\tfrac12\bigl[\Acal(\partial_ty,y)+\Acal(y,\partial_ty)\bigr]
      =\Acal(y,\partial_ty)=\ip{\Acal y}{\partial_ty}.
\end{equation}
Adding \eqref{eq:e1prime} and \eqref{eq:e2prime}, the terms
$\mp\ip{\Acal y}{\partial_ty}$ cancel:
\begin{equation}\label{eq:e12}
  e_1'+e_2'=-\ip{\abs u^pu}{\partial_ty}.
\end{equation} Since $\partial_tu=\partial_ty+\partial_tz$,
Pointwise in $x$, with
$F(a):=\abs a^{p+2}$ one has $F\in C^1(\R)$ for $p>0$ with
$F'(a)=(p+2)\abs a^{p}a$, so
$\partial_t\abs u^{p+2}=(p+2)\abs u^pu\,\partial_tu$ and
\begin{equation}\label{eq:e3prime}
  e_3'=\ip{\abs u^pu}{\partial_tu}.
\end{equation}
Splitting it according to the reduction $u=y+z$,
\begin{equation}\label{eq:e3}
  e_3'=\ip{\abs u^pu}{\partial_ty}+\ip{\abs u^pu}{\partial_tz}.
\end{equation}
Adding
\eqref{eq:e12} and \eqref{eq:e3}, the terms
$\mp\ip{\abs u^pu}{\partial_ty}$ cancel in turn, and only the coupling of
the nonlinearity with the \emph{velocity of the stochastic convolution}
survives:
\begin{equation}\label{eq:e-identity}
  e'(t)=\bigl\langle \abs{u}^pu,\ \partial_tz\bigr\rangle ,
  \qquad t\in[0,\underline T],
\end{equation}
equivalently, in integrated form,
$e(t)=e(0)+\int_0^t\ip{\abs u^pu}{\partial_tz}\,dr$.

\emph{Step 3 (global bound).} By H\"older with exponents $\frac{p+2}{p+1}$
and $p+2$,
\[
  \abs{e'(t)}\;\le\;\normLp{u}{p+2}^{p+1}\,\normLp{\partial_tz}{p+2}
  \;=\;\bigl((p+2)\,e_3(t)\bigr)^{\frac{p+1}{p+2}}\psi(t)
  \;\le\;C_\star\bigl(1+e(t)\bigr)\psi(t),
\]
where $e_3:=\frac1{p+2}\normLp{u}{p+2}^{p+2}\le e$ and we used
$a^{\frac{p+1}{p+2}}\le1+a$.  Gronwall's lemma gives
\eqref{eq:pathwise-gronwall}. Since $\psi\in C([0,T])$ a.s.\ by
Lemma~\ref{lem:z}(iii), $\sup_{[0,T]}e<\infty$ a.s.; as
$\normV{y}^2\le2e$ and $\normH{\partial_ty}^2\le2e$, the blow-up
alternative of Step~1 cannot occur and $\underline T$ can be continued to
$+\infty$. This proves \eqref{eq:global-class}, using $u=y+z$ and
Lemma~\ref{lem:z}.

\emph{Step 4 (adaptedness).} By Steps~1--3 the deterministic solution map
\[
  \Theta:\ \Vcal\times C([0,T];V)\longrightarrow C([0,T];\Vcal),\qquad
  (U_0,V_0,\zeta)\longmapsto (y,\partial_ty),
\]
is well defined and locally Lipschitz (the fixed-point estimates
\eqref{eq:Nlip-defoc}--\eqref{eq:Ngrowth-defoc} depend on $\zeta$ only
through $\sup_t\normV{\zeta}$). Since $(U_0,V_0)$ is
$\mathcal F_0$-measurable and $t\mapsto z(t)$ is
$\{\mathcal F_t\}$-adapted and continuous in $V$, and since $\Theta$
restricted to $[0,t]$ depends on $\zeta$ only through $\zeta|_{[0,t]}$, the
process $y=\Theta(U_0,V_0,z)$ is $\{\mathcal F_t\}$-adapted, and so is
$u=y+z$. 

\emph{Step 5 (uniqueness and moments).} Pathwise uniqueness is immediate
from Step~1 (the fixed point is unique) and coincides with
Proposition~\ref{prop:uniqueness} below. For the moment bound,
Lemma~\ref{lem:z}(iii) and Fernique's theorem give
$\E\exp\bigl(\lambda\sup_{t\le T}\psi(t)\bigr)<\infty$ for every
$\lambda>0$; since $\int_0^T\psi\le T\sup_{t\le T}\psi$,
\eqref{eq:pathwise-gronwall} and H\"older gives the stated conclusion.
\end{proof}

\subsection{Strong Galerkin convergence and the energy identity}

Write $\Acal_N:=\theta\Lambda^{(N)}+\beta M^{(N)}=A^{(N)}$ for the projected
operator on $H_N$ (Lemma~\ref{lem:Mproperties}), $S_N$ for the corresponding
unitary group on $\Vcal_N:=H_N\times H_N$ normed by
$\ip{\Acal_NU}{U}+\abs V^2$, and
\begin{equation}\label{eq:ZN}
  Z^N(t):=\bigl(z^N,\partial_tz^N\bigr)(t)
  :=\int_0^tS_N(t-r)\,\Pi^N\Gbb\,dW_Q(r).
\end{equation}

\begin{lemma}\label{lem:trotter}
$\bigcup_N H_N$ is a form core for $\Acal$; hence $\Acal_N\Pi^N\to\Acal$ in
the strong resolvent sense and, by the Trotter--Kato theorem,
\[
  S_N(t)\Pi^N X\longrightarrow S(t)X\quad\text{in }\Vcal,
  \qquad\text{uniformly for }t\text{ in compact subsets of }\R,
\]
for every $X\in\Vcal$. Moreover, uniformly in $N$ and $\rho\in[0,1]$,
\begin{equation}\label{eq:heinzN}
  \theta^{\rho/2}\bigl\lvert(\Lambda^{(N)})^{\rho/2}U\bigr\rvert
  \;\le\;\bigl\lvert\Acal_N^{\rho/2}U\bigr\rvert
  \;\le\;(\theta+\beta C_s)^{\rho/2}
  \bigl\lvert(\Lambda^{(N)})^{\rho/2}U\bigr\rvert,
  \qquad U\in\R^N .
\end{equation}
\end{lemma}

\begin{proof}
Write $b(v,w):=\Acal(v,w)+\ip{v}{w}$ and $\norm{v}_b^2:=b(v,v)$, a Hilbert
norm on $V$ equivalent to $\normV\cdot$ by Poincar\'e, and set
\[
  R_N:=(\Acal_N+I)^{-1}\Pi^N\ \ (\text{zero on }H_N^{\perp}),
  \qquad R:=(\Acal+I)^{-1},
\]
the resolvents at $-1$ of the forms $a_N:=a|_{H_N\times H_N}$ (extended by
$+\infty$ off $H_N$) and $a:=\Acal(\cdot,\cdot)$.
Fix $f\in H$ and put $u:=Rf$, $u_N:=R_Nf$, so that $b(u,v)=\ip{f}{v}$ for all
$v\in V$ and $b(u_N,v_N)=\ip{f}{v_N}$ for all $v_N\in H_N$. Subtracting gives
$b(u-u_N,v_N)=0$ on $H_N$, so $u_N$ is the $b$-orthogonal projection of $u$
onto $H_N$ and
\begin{equation}\label{eq:cea}
  \norm{u-u_N}_b=\operatorname{dist}_b(u,H_N).
\end{equation}
Testing the first identity with $v=u_N$ gives
$b(u,u_N)=\ip{f}{u_N}=b(u_N,u_N)$, whence
\begin{equation}\label{eq:energy-gap}
  \norm{u-u_N}_b^2=b(u,u)-b(u_N,u_N)=\ip{f}{u-u_N}_H .
\end{equation}
By \eqref{eq:cea} and \eqref{eq:energy-gap},
\[
  u_N\to u\ \text{in }H
  \iff u_N\to u\ \text{in }V
  \iff \operatorname{dist}_b(u,H_N)\to0 .
\]
Since $\operatorname{Ran}R=D(\Acal)$ is dense in $V$ and the $H_N$ increase, the last
condition holds for every $f\in H$ if and only if $\bigcup_NH_N$ is dense in
$V$ for $\norm\cdot_b$, i.e.\ is a form core for $\Acal$. Density in $H$
alone is insufficient: by \cite[Thm.~VIII.3.11]{Kato1995} the monotone limit
of the $a_N$ is the closure of $a|_{\bigcup_NH_N}$, which equals $a$ only
under form-norm density.\\
 For $u=\sum_ku_kh_k\in H^1_0(\Ocal)$
one has $\norm{\nabla(u-\Pi^Nu)}^2=\sum_{k>N}\lambda_ku_k^2\to0$, since
$\sum_k\lambda_ku_k^2=\norm{\nabla u}^2<\infty$ and $\Pi^N$ is orthogonal for
both $\normH\cdot$ and $\norm{\nabla\cdot}$; by \eqref{eq:V-equiv} this gives
$\Pi^Nu\to u$ in $\norm\cdot_b$. 
Since $\ip{\Acal_NU}{U}=\normV{U}^2$ for
$U\in H_N$, the space $\Vcal_N$ carries the norm inherited from $\Vcal$ and
is a closed subspace of it. Strong resolvent convergence of $\Acal_N\Pi^N$
therefore implies that of the skew-adjoint generators
$\Gcal_N=\begin{psmallmatrix}0&I\\-\Acal_N&0\end{psmallmatrix}$ to $\Gcal$ on
$\Vcal$, and the Trotter--Kato theorem
\cite[Thm.~IX.2.16]{Kato1995} yields the stated convergence of the unitary
groups, uniformly for $t$ in compact sets. 
For the bound apply the Heinz--Kato inequality
($0\le P\le Q$ implies $\abs{P^{\rho/2}U}\le\abs{Q^{\rho/2}U}$ for
$\rho\in[0,1]$) to $\theta\Lambda^{(N)}\le A^{(N)}\le(\theta+\beta
C_s)\Lambda^{(N)}$ from Lemma~\ref{lem:Mproperties}; uniformity in $N$ holds
because those constants are.
\end{proof}

\begin{proposition}[Strong convergence of the nonlinear Galerkin approximation]
\label{prop:strongconv}
Under the hypotheses of Theorem~\ref{thm:global-dd}, for every $T>0$
\begin{equation}\label{eq:ZNconv}
  \E\sup_{t\le T}\norm{Z^N(t)-Z(t)}_{\Vcal}^2\longrightarrow0,
  \qquad
  \sup_N\ \E\sup_{t\le T}\normLp{\partial_tz^N(t)}{p+2}^2<\infty,
\end{equation}
and consequently
\begin{equation}\label{eq:uNconv}
  \sup_{t\le T}\Bigl(\normV{u^N(t)-u(t)}+\normH{v^N(t)-v(t)}\Bigr)
  \longrightarrow0
  \qquad\PP\text{-a.s.\ and in }L^q(\Omega),\ q<\infty .
\end{equation}
\end{proposition}
\begin{remark}
    The main difficulty of this theorem, is that the classical Laplacian eigenbasis diagonalizes neither the mixed operator $\mathcal{A}$ nor its fractional component. Thus the standard spectral Galerkin argument based on simultaneous diagonalization is unavailable. The proof instead relies on form-core convergence and strong resolvent convergence of $\mathcal{A}_N$, followed by a stability argument for the nonlinear equation.
\end{remark}
\begin{proof}
By the group trick
\eqref{eq:group-trick}, $Z^N(t)=S_N(t)Y^N(t)$ with
$Y^N(t)=\int_0^tS_N(-r)\Pi^N\Gbb\,dW_Q(r)$, and $Y^N-Y$ is a continuous
$\Vcal$-valued martingale, so by the It\^o isometry
\begin{equation}\label{eq:isom-diff}
  \E\norm{Y^N(t)-Y(t)}^2_{\Vcal}
  =\int_0^t g_N(r)\,dr,
  \qquad
  g_N(r):=\sum_{k\ge1}a^N_k(r),
\end{equation}
where
$a^N_k(r):=\bigl\|\bigl(S_N(-r)\Pi^N-S(-r)\bigr)\Gbb Q^{1/2}h_k
\bigr\|^2_{\Vcal}$.
 Since $Q^{1/2}h_k=\sigma\lambda_k^{-\gamma}h_k$
and $\Gbb g=(0,g)$ with $\norm{\Gbb g}_{\Vcal}=\normH g$, the vector
$\Gbb Q^{1/2}h_k=\sigma\lambda_k^{-\gamma}(0,h_k)$ lies in $\{0\}\times H$,
where $\Pi^N$ acts as the $H$-orthogonal projection; hence
$\norm{\Pi^N\Gbb Q^{1/2}h_k}_{\Vcal}\le\sigma\lambda_k^{-\gamma}$. As $S(-r)$
and $S_N(-r)$ are isometries of $\Vcal$ and $\Vcal_N$ respectively
(see the proof of Lemma~\ref{lem:trotter}), the triangle inequality gives
\begin{equation}\label{eq:dominating}
  a^N_k(r)\;\le\;\bigl(\sigma\lambda_k^{-\gamma}
  +\sigma\lambda_k^{-\gamma}\bigr)^2
  \;=\;4\sigma^2\lambda_k^{-2\gamma}
  \qquad\text{for all }N\in\N,\ r\in[0,T],
\end{equation}
a bound independent of both $N$ and $r$, 
and it is summable in $k$ with
$\sum_k4\sigma^2\lambda_k^{-2\gamma}=4\sigma^2\kappa<\infty$ by
Assumption~\ref{ass:standing}(i).\\
Fix $r$. For each $k$, Lemma~\ref{lem:trotter} applied to the fixed vector
$X=\Gbb Q^{1/2}h_k\in\Vcal$ gives $a^N_k(r)\to0$ as $N\to\infty$. By
\eqref{eq:dominating} the sequence $(a_k^N(r))_k$ is dominated by the
summable sequence $(4\sigma^2\lambda_k^{-2\gamma})_k$, uniformly in $N$;
dominated convergence for series therefore yields $g_N(r)\to0$ for every
$r\in[0,T]$.

 Summing
\eqref{eq:dominating} over $k$ gives $0\le g_N(r)\le4\sigma^2\kappa$ for all
$N$ and $r$, and the constant function $4\sigma^2\kappa$ is integrable on the
finite interval $[0,T]$. Since $g_N\to0$ pointwise by the previous step,
dominated convergence gives $\int_0^Tg_N(r)\,dr\to0$, i.e.
$\E\norm{Y^N(T)-Y(T)}^2_{\Vcal}\to0$ by \eqref{eq:isom-diff}. Doob's
inequality upgrades this to $\E\sup_{t\le T}\norm{Y^N-Y}^2_{\Vcal}\to0$, and
the first half of \eqref{eq:ZNconv} follows because $S_N$ and $S$ are
isometries.

The second half of \eqref{eq:ZNconv} is the computation of
Lemma~\ref{lem:z}(iii) performed on $H_N$, where \eqref{eq:heinzN} makes the
constants $N$-independent and
$\lvert(\Lambda^{(N)})^{\delta/2}\Pi^Nu\rvert\le\norm{(-\Delta)^{\delta/2}u}$.

 Write $u^N=y^N+z^N$, where $y^N$
solves the projected version of \eqref{eq:reduced-defoc}, and set
$Y:=(y,\partial_ty)$, $Y^N:=(y^N,\partial_ty^N)$. Fix $\omega$ outside a null
set and put
\[
  R:=1+\sup_N\sup_{t\le T}\norm{Y^N(t)}_{\Vcal}
   +\sup_{t\le T}\norm{Y(t)}_{\Vcal}
   +\sup_N\sup_{t\le T}\norm{Z^N(t)}_{\Vcal}
   +\sup_{t\le T}\norm{Z(t)}_{\Vcal},
\]
which is finite because \eqref{eq:pathwise-gronwall} holds for $y^N$ with a
constant independent of $N$, by the second half of \eqref{eq:ZNconv}. Let
$L=L(R)$ be the Lipschitz constant of $N(\cdot,\zeta)$ on the ball of radius
$3R$ furnished by \eqref{eq:Nlip-defoc}, and choose the step
\begin{equation}\label{eq:glue-step}
  h:=\min\bigl\{(2L)^{-1},\,T\bigr\},
  \qquad M:=\lceil T/h\rceil ,
\end{equation}
both independent of $N$ and of the starting time.

On an interval $[t_0,t_0+h]\subset[0,T]$, subtracting the two Duhamel
formulas gives, for $t$ in that interval,
\[
  Y^N(t)-Y(t)
  =S_N(t-t_0)\bigl[Y^N(t_0)-\Pi^NY(t_0)\bigr]
  +\Ecal_N(t)
  -\int_{t_0}^tS_N(t-r)\Pi^N\Gbb
   \bigl[N(y^N,z^N)-N(y,z)\bigr]dr,
\]
where $\Ecal_N$ collects the terms in which the group $S_N\Pi^N$ is replaced
by $S$, namely $\bigl(S_N(t-t_0)\Pi^N-S(t-t_0)\bigr)Y(t_0)$ and
$\int_{t_0}^t\bigl(S_N(t-r)\Pi^N-S(t-r)\bigr)\Gbb N(y,z)(r)\,dr$. Both
$\{Y(t_0):t_0\in[0,T]\}$ and $\{\Gbb N(y,z)(r):r\in[0,T]\}$ are compact
subsets of $\Vcal$, being continuous images of $[0,T]$; since the operators
$S_N(\tau)\Pi^N-S(\tau)$ are uniformly bounded and converge strongly,
uniformly for $\tau$ in compacts, the convergence is uniform on compact
subsets of $\Vcal$, so
\[
  \varepsilon_N:=\sup_{t_0\le T}\ \sup_{t\in[t_0,t_0+h]}
  \norm{\Ecal_N(t)}_{\Vcal}
  +\sup_{t\le T}\norm{Z^N(t)-Z(t)}_{\Vcal}
  \longrightarrow0 .
\]
Using $\norm{S_N(\cdot)}=1$, $\norm{\Gbb g}_{\Vcal}=\normH g$,
\eqref{eq:Nlip-defoc} and $hL\le\frac12$, we obtain
\[
  \sup_{[t_0,t_0+h]}\norm{Y^N-Y}_{\Vcal}
  \le\norm{Y^N(t_0)-Y(t_0)}_{\Vcal}+2\varepsilon_N
  +\tfrac12\sup_{[t_0,t_0+h]}\norm{Y^N-Y}_{\Vcal},
\]
that is, $\sup_{[t_0,t_0+h]}\norm{Y^N-Y}_{\Vcal}
\le2\norm{Y^N(t_0)-Y(t_0)}_{\Vcal}+4\varepsilon_N$. Applying this on the
consecutive intervals $[jh,(j+1)h]$, $j=0,\dots,M-1$, and iterating,
\begin{equation}\label{eq:glue}
  \sup_{t\le T}\norm{Y^N(t)-Y(t)}_{\Vcal}
  \le2^{M}\norm{\Pi^N(U_0,V_0)-(U_0,V_0)}_{\Vcal}+2^{M+2}\varepsilon_N
  \longrightarrow0 ,
\end{equation}
since $M$ depends only on $L(R)$ and $T$, not on $N$. (we are following the continuation-by-finitely-many-steps argument for semilinear evolution
equations, see \cite[Thm.~6.1.4]{Pazy1983} or
\cite[Ch.~4]{CazenaveHaraux1998}). Together with the first half of
\eqref{eq:ZNconv} and $u^N=y^N+z^N$ this proves the almost sure convergence
in \eqref{eq:uNconv}. Uniform integrability, hence $L^q(\Omega)$
convergence, follows from \eqref{eq:pathwise-gronwall} and the Fernique
bound of Lemma~\ref{lem:z}.
\end{proof}

\begin{theorem}[It\^o energy identity]\label{thm:energy-identity}
Under the hypotheses of Theorem~\ref{thm:global-dd}, the unique solution
satisfies, $\PP$-a.s.\ for all $t\ge0$,
\begin{equation}\label{eq:energy-identity-fixed}
  \Ecal(t)=\Ecal(0)+\frac{\kappa}{2}\,t
   +\sigma\sum_{k\ge1}\lambda_k^{-\gamma}\int_0^t\ip{v(r)}{h_k}\,dw_k(r),
  \qquad
  \Ecal:=\tfrac12\normH{v}^2+\tfrac12\normV{u}^2
  +\tfrac1{p+2}\normLp{u}{p+2}^{p+2}.
\end{equation}
In particular $t\mapsto\Ecal(t)$ is a.s.\ continuous, and
$u\in C([0,T];V)$, $v\in C([0,T];H)$ a.s.
\end{theorem}

\begin{proof}
Lemma~\ref{lem:galerkin} (It\^o's formula in $\R^{2N}$) gives, for every
$N$,
\[
  \Ecal^N(t)=\Ecal^N(0)+\frac{\kappa_N}{2}t
   +\sigma\sum_{k\le N}\lambda_k^{-\gamma}\int_0^t v^N_k(r)\,dw_k(r),
   \qquad
   \kappa_N=\sigma^2\!\!\sum_{k\le N}\!\lambda_k^{-2\gamma}\uparrow\kappa .
\]
By Proposition~\ref{prop:strongconv} the left-hand side converges to
$\Ecal(t)$ uniformly on $[0,T]$, a.s.\ and in $L^1(\Omega)$: indeed
$\normV{u^N}\to\normV u$ and $\normH{v^N}\to\normH v$ by \eqref{eq:uNconv},
and $\normLp{u^N}{p+2}\to\normLp{u}{p+2}$ by
$V\hookrightarrow L^{p+2}$. Also $\Ecal^N(0)\to\Ecal(0)$ by
Lemma~\ref{lem:initial-energy}. For the martingale term,
\[
  \E\Bigl\lvert\sum_{k\le N}\lambda_k^{-\gamma}\!\int_0^t\!
  \bigl(v^N_k-\ip{v}{h_k}\bigr)dw_k
   -\!\!\sum_{k>N}\lambda_k^{-\gamma}\!\int_0^t\!\ip{v}{h_k}dw_k
   \Bigr\rvert^2
  \le \lambda_1^{-2\gamma}\E\!\int_0^t\!\normH{v^N-v}^2
   +\!\sum_{k>N}\!\lambda_k^{-2\gamma}\E\!\int_0^t\!\normH{v}^2,
\]
and both terms vanish as $N\to\infty$ by \eqref{eq:uNconv} and
$\sum_k\lambda_k^{-2\gamma}<\infty$. Passing to the limit gives
\eqref{eq:energy-identity-fixed} for each fixed $t$, and for all $t$
simultaneously by continuity of both sides. Continuity of $\Ecal$ together
with the a.s.\ weak continuity of $u$ in $V$ and $v$ in $H$ upgrades weak to
strong continuity by the standard norm-plus-weak argument.
\end{proof}

\subsection{Pathwise uniqueness}

\begin{proposition}\label{prop:uniqueness}
Let Assumption~\ref{ass:standing} hold and assume in addition
\begin{equation}\label{eq:uniqueness-range}
0<p\le\frac{2}{d-2}\quad (d\ge3),\qquad
\text{no restriction if } d\le 2 .
\end{equation}
If $(u_1,v_1)$ and $(u_2,v_2)$ are energy solutions on $[0,T]$ with the same
initial data, then
$\PP\bigl(u_1(t)=u_2(t)\ \forall t\in[0,T]\bigr)=1$.
\end{proposition}

\begin{proof}
Set $w:=u_1-u_2$ and $z:=v_1-v_2$. Because the noise is \emph{additive}, it
cancels in the difference: $(w,z)$ satisfies, $\PP$-a.s., the
\emph{deterministic} system $\partial_tw=z$,
$\partial_tz+\Acal w=-\bigl(\abs{u_1}^pu_1-\abs{u_2}^pu_2\bigr)$ with
$w(0)=z(0)=0$. Testing with $z$ (which is admissible after a standard
mollification in time, cf.\ \cite[Ch.~3]{LionsMagenes1972}) and setting
$e(t):=\frac12\norm{z(t)}^2+\frac12\normV{w(t)}^2$,
\[
e'(t)=-\ip{\abs{u_1}^pu_1-\abs{u_2}^pu_2}{z}
\;\le\;\norm{\abs{u_1}^pu_1-\abs{u_2}^pu_2}\,\norm{z} .
\]
By the elementary inequality
$\bigl|\abs a^pa-\abs b^pb\bigr|\le(p+1)(\abs a+\abs b)^p\abs{a-b}$ and
H\"older with $\frac12=\frac pr+\frac1q$,
\[
\norm{\abs{u_1}^pu_1-\abs{u_2}^pu_2}
\le (p+1)\,\normLp{\abs{u_1}+\abs{u_2}}{r}^{p}\,\normLp{w}{q}.
\]
Choose $q=2^{*}=\frac{2d}{d-2}$; then
$\frac pr=\frac12-\frac{d-2}{2d}=\frac1d$, i.e.\ $r=pd$, and $r\le2^{*}$
exactly when $p\le\frac{2}{d-2}$, which is \eqref{eq:uniqueness-range}.
Hence by Sobolev
\[
e'(t)\le C\bigl(\normV{u_1}^p+\normV{u_2}^p\bigr)\normV{w}\norm{z}
\le K(\omega)\,e(t),\qquad
K(\omega):=C\sup_{t\le T}\bigl(\normV{u_1}^p+\normV{u_2}^p\bigr)
<\infty\ \text{a.s.}
\]
Gronwall and $e(0)=0$ give $e\equiv0$.
\end{proof}

\begin{corollary}[Global well-posedness]\label{cor:gwp}
Under Assumption~\ref{ass:standing}, \eqref{eq:local-range} and
\eqref{eq:gamma-strong}, problem \eqref{stochastic_problem} has a unique
global energy solution, and the whole Galerkin sequence  converges to it in the sense of
Proposition~\ref{prop:strongconv}.
\end{corollary}


\section{Malliavin Regularity and Existence of Density} \label{sec:malliavin}

In this section, we study the Malliavin regularity of the solution for the defocusing case ($\varepsilon = +1$). Our main goal is to prove that the solution $u(t, x)$ is Malliavin differentiable and that, under suitable spatial regularity conditions, its probability law evaluated at a specific point admits a density with respect to the Lebesgue measure on $\R$. We rely on the standard framework of Malliavin calculus for stochastic partial differential equations (see~\cite{SanzSole2005} for an application of Malliavin calculus to SPDEs and~\cite{nualart2006malliavin} for a comprehensive study of Malliavin calculus).

Let $\mathbb{D}^{1,2}(H)$ denote the standard Malliavin Sobolev space of random variables with values in the pivot space $H = L^2(\Ocal)$. We denote the Malliavin derivative with respect to the driving noise $W_Q$ at time $r$ and in the direction of the eigenfunction $h_k$ as $D_{r,k}$ \cite{SanzSole2005}.

\subsection{The Variational Equation and Malliavin Differentiability}
Applying the Malliavin derivative operator $D_{r,k}$ formally to the integral representation of our system, we obtain the following linear stochastic evolution equation with random coefficients. For $t \ge r$, the derivative $D_{r,k} u(t)$ satisfies:
\begin{align}
D_{r,k} u(t) &= \int_r^t D_{r,k} v(s) \, ds, \label{eq:mal_u} \\
D_{r,k} v(t) &= \sigma \lambda_k^{-\gamma} h_k - \int_r^t \left[ \Acal (D_{r,k} u(s)) + f'(u(s)) D_{r,k} u(s) \right] \, ds, \label{eq:mal_v}
\end{align}
where $f'(u) = (p+1)\abs{u}^p$. For $t < r$, we set $D_{r,k} u(t) = 0$ and $D_{r,k} v(t) = 0$.

\begin{theorem} \label{thm:malliavin_diff}
Let the conditions of the defocusing case (Assumption~\ref{ass:standing}) hold. Then, for any $t > 0$, the solution $(u(t), v(t))$ belongs to $\mathbb{D}^{1,2}(V) \times \mathbb{D}^{1,2}(H)$.
\end{theorem}

\begin{proof}
We introduce the energy functional for the Malliavin derivative in the direction $k$:
\begin{equation}
\Ecal_{D,k}(t, r) = \frac{1}{2} \normH{D_{r,k} v(t)}^2 + \frac{1}{2} \normV{D_{r,k} u(t)}^2 = \frac{1}{2} \normH{D_{r,k} v(t)}^2 + \frac{1}{2} \ip{\Acal (D_{r,k} u(t))}{D_{r,k} u(t)} \label{eq:mal_energy} \footnotemark
\end{equation}
\footnotetext{Physically, the functional $\Ecal_{D,k}(t, r)$ measures the sum of the kinetic energy (associated with the velocity of the Malliavin derivative $D_{r,k} v(t)$) and the potential energy (associated with the deformation of $D_{r,k} u(t)$ in the energy space $V$) between the perturbation time $r$ and the observation time $t$.}

Since the noise acts additively in the original equation, the variational equations \eqref{eq:mal_u} and \eqref{eq:mal_v} contain no stochastic integrals for $t > r$. Thus, we can compute the classical time derivative of $\Ecal_{D,k}(t, r)$ along the trajectories:
\begin{align*}
\frac{d}{dt} \Ecal_{D,k}(t, r) &= \inner{D_{r,k} v(t)}{D_{r,k} v'(t)} + \ip{\Acal (D_{r,k} u(t))}{D_{r,k} u'(t)} \\
&= \inner{D_{r,k} v(t)}{-\Acal(D_{r,k} u(t)) - f'(u(t))D_{r,k} u(t)} + \ip{\Acal(D_{r,k} u(t))}{D_{r,k} v(t)} \\
&= - \int_{\Ocal} f'(u(t,x)) \abs{D_{r,k} u(t,x)}^2 \, dx.
\end{align*}
Because $p > 0$, the derivative of the nonlinearity is non-negative: $f'(u(t,x)) = (p+1)\abs{u(t,x)}^p \ge 0$. Consequently, the time derivative of the energy is negative or zero, meaning that the energy is non-increasing for $t \ge r$:
\begin{equation}
\frac{d}{dt} \Ecal_{D,k}(t, r) \le 0.
\end{equation}
Integrating this inequality, we can bound the energy by its initial value at $t=r$:
\begin{equation}
\Ecal_{D,k}(t, r) \le \Ecal_{D,k}(r, r) = \frac{1}{2} \normH{D_{r,k} v(r)}^2 = \frac{1}{2} \sigma^2 \lambda_k^{-2\gamma},
\end{equation}
where we used $\normH{h_k} = 1$. Taking expectations and summing over $k \in \N$, the total Malliavin energy is bounded by:
\begin{equation}
\E \left[ \sum_{k \in \N} \Ecal_{D,k}(t, r) \right] \le \frac{1}{2} \sigma^2 \sum_{k \in \N} \lambda_k^{-2\gamma} = \frac{1}{2} \kappa,
\end{equation}
which is finite by the trace-class noise condition in Assumption~\ref{ass:standing}(i) ($\kappa = \Tr Q < \infty$). This uniform bound completes the proof that $u(t) \in \mathbb{D}^{1,2}(V)$ and $v(t) \in \mathbb{D}^{1,2}(H)$.
\end{proof}

\subsection{Existence of the Probability Density}
We now focus on the existence of a density for the random variable evaluated at a single spatial point. We assume the spatial dimension is $d=1$ and the fractional order $s \in (1/2, 1)$. Under these conditions, the fractional Sobolev embedding $\widetilde H^s(\Ocal) \hookrightarrow C^{0, s-1/2}(\bar{\Ocal})$ holds (see Section~\ref{sec:inference}). This guarantees that the evaluation functional $E_{x_0}(u) = u(x_0)$ is continuous on $V = H^1_0(\Ocal)$, and thus the random variable $u(t, x_0)$ belongs to $\mathbb{D}^{1,2}(\R)$.

\begin{theorem} \label{thm:density}
Let $d=1$, $s \in (1/2, 1)$, and fix $(t, x_0) \in (0, T] \times \Ocal$. Then, the probability law of the random variable $u(t, x_0)$ is absolutely continuous with respect to the Lebesgue measure on $\R$.
\end{theorem}

\begin{proof}
According to the Bouleau-Hirsch criterion (see \cite{SanzSole2005}, \cite[Section 2.1]{nualart2006malliavin}), it is sufficient to prove that the Malliavin variance of $u(t, x_0)$ is strictly positive almost surely:
\begin{equation}\label{eq:taylor_u}
\Sigma_t = \int_0^t \norm{ D_r u(t, x_0) }_{\fU}^2 \, dr = \int_0^t \sum_{k=1}^{\infty} \abs{D_{r,k} u(t, x_0)}^2 \, dr > 0 \quad \PP\text{-a.s.}
\end{equation}
We analyze the short-time asymptotic behavior of the derivative when $r \uparrow t$. To avoid differentiating unbounded operators and to rigorously track the remainder, we express $D_{r,k}u(t)$ via its mild formulation in the energy space $V$. For $t \ge r$, Duhamel's principle yields:
\begin{equation}
D_{r,k}u(t) = I_1(t,r) + I_2(t,r),
\end{equation}
where the linear perturbation term is
\begin{equation}
I_1(t,r) = \mathcal{A}^{-1/2}\sin((t-r)\mathcal{A}^{1/2}) \left[ \sigma\lambda_k^{-\gamma}h_k \right],
\end{equation}
and the nonlinear cross term is
\begin{equation}
I_2(t,r) = - \int_r^t \mathcal{A}^{-1/2}\sin((t-s)\mathcal{A}^{1/2}) \left[ f'(u(s))D_{r,k}u(s) \right] ds.
\end{equation}

We evaluate the expansion in $V$. For the linear term, by the fundamental theorem of calculus and the strong continuity of the group $\cos(\tau\mathcal{A}^{1/2})$ on $V$, we have:
\begin{equation}
\lim_{\tau \to 0} \norm{ \frac{1}{\tau} I_1(r+\tau, r) - \sigma\lambda_k^{-\gamma}h_k }_V = 0.
\end{equation}
Thus, $I_1(t,r) = (t-r)\sigma\lambda_k^{-\gamma}h_k + o_V(t-r)$.

For the nonlinear remainder $I_2(t,r)$, using the unitary bound $\norm{\mathcal{A}^{-1/2}\sin(\tau \mathcal{A}^{1/2}) g}_V \le \norm{g}_H$, we obtain:
\begin{equation}
\norm{I_2(t,r)}_V \le \int_r^t \norm{f'(u(s))D_{r,k}u(s)}_H ds \le \int_r^t \norm{f'(u(s))}_{L^\infty} \norm{D_{r,k}u(s)}_H ds.
\end{equation}
Since $d=1$, the Sobolev embedding $V \hookrightarrow L^\infty(\Ocal)$ ensures that $\norm{f'(u(s, \omega))}_{L^\infty}$ is bounded by a random constant $C(\omega)$ on $[0,T]$. Furthermore, $D_{r,k}u(r) = 0$, and by the pathwise continuity of the Malliavin derivative in $V$ (Theorem 6.1), we have $\norm{D_{r,k}u(s)}_V \to 0$ as $s \downarrow r$. Consequently, the integrand vanishes as $s \to r$, yielding:
\begin{equation}
\norm{I_2(t,r)}_V \le C(\omega) \int_r^t \norm{D_{r,k}u(s)}_V ds = o_V(t-r) \quad \PP\text{-a.s.}
\end{equation}

Combining these estimates, we obtain the robust asymptotic expansion in the energy space:
\begin{equation}
\norm{ D_{r,k}u(t) - (t-r)\sigma\lambda_k^{-\gamma}h_k }_V = o_V(t-r) \quad \PP\text{-a.s.}
\end{equation}
To transition to the pointwise evaluation, we use the fractional Sobolev embedding for $d=1$ and $s \in (1/2, 1)$. Since $V \hookrightarrow C(\bar{\Ocal})$, there exists a constant $C_{emb} > 0$ such that $\sup_{x \in \Ocal} \abs{v(x)} \le C_{emb}\norm{v}_V$. Evaluating at $x_0 \in \Ocal$, the norm convergence rigorously translates to the pointwise expansion:
\begin{equation}
D_{r,k} u(t, x_0) = \sigma \lambda_k^{-\gamma} h_k(x_0) (t-r) + o(t-r) \quad \PP\text{-a.s.}
\end{equation}

Since $\{h_k\}_{k \ge 1}$ is a complete orthonormal basis in $L^2(\Ocal)$, there exists at least one integer $k_0$ such that $h_{k_0}(x_0) \neq 0$. For this specific $k_0$, there is a sufficiently small $\epsilon = \epsilon(k_0, x_0, p, \omega) > 0$ such that for all $r \in [t-\epsilon, t]$, the linear term dominates the residual. This implies:
\begin{equation}
\abs{D_{r,k_0} u(t, x_0)}^2 \ge C_{k_0, x_0} (t-r)^2, \label{eq:lower_bound_mode}
\end{equation}
where $C_{k_0, x_0} = \frac{1}{2} \sigma^2 \lambda_{k_0}^{-2\gamma} \abs{h_{k_0}(x_0)}^2 > 0$. Finally, we lower bound the total Malliavin variance by restricting the integral to the interval $[t-\epsilon, t]$ and considering only the $k_0$-th mode:
\begin{equation}
\Sigma_t \ge \int_{t-\epsilon}^t \abs{D_{r,k_0} u(t, x_0)}^2 \, dr \ge \int_{t-\epsilon}^t C_{k_0, x_0} (t-r)^2 \, dr = C_{k_0, x_0} \frac{\epsilon^3}{3} > 0 \quad \PP\text{-a.s.}
\end{equation}
Because this lower bound is strictly positive almost surely (since $\epsilon > 0$ a.s.), the Bouleau-Hirsch criterion is satisfied, and the existence of the density is proved.
\end{proof}

\begin{corollary} \label{cor:embedding_limit}
For $d \ge 2$ or $s \le 1/2$, the pointwise evaluation $u(t, x_0)$ is not covered by the current embedding-based framework, as $H^s(\Ocal)$ fails to embed into $C(\bar{\Ocal})$. This represents a fundamental limitation of the Sobolev embedding rather than a failure of the analytical proof, necessitating a more sophisticated analysis (such as the use of spatially averaged observations) to establish density properties.
\end{corollary}

\begin{remark}[Extension to Higher Dimensions and Spatially Averaged Observations] \label{rem:malliavin_higher_d}
In dimensions $d \ge 2$, the energy space $V = H^1_0(\Ocal)$ fails to embed into the space of continuous functions $C(\bar{\Ocal})$ since the critical Sobolev threshold is crossed. Specifically, in the two-dimensional case ($d=2$), we only have the embedding $H^1_0(\Ocal) \hookrightarrow L^q(\Ocal)$ for all $q < \infty$, meaning that functions in the energy space may possess local logarithmic singularities. Consequently, the pointwise evaluation functional $E_{x_0}(u) = u(x_0)$ is no longer continuous on $V$, and the pointwise value $u(t, x_0)$ is not well-defined as a random variable in $\mathbb{D}^{1,2}(\R)$. 

To circumvent this geometric limitation and establish the existence of a density in higher dimensions, one must instead consider spatially averaged observations of the solution. Let $\phi \in C_c^\infty(\Ocal)$ (or more generally $\phi \in H$) be a smooth, non-negative test function representing a spatial measurement kernel. We define the averaged random variable:
\begin{equation}
F_\phi = \ip{u(t)}{\phi} = \int_{\Ocal} u(t, x) \phi(x) \, dx.
\end{equation}
Since the mapping $u \mapsto \ip{u}{\phi}$ is a continuous linear functional on the pivot space $H = L^2(\Ocal)$, it is automatically continuous on the energy space $V$ in any dimension $d \ge 1$. Hence, $F_\phi \in \mathbb{D}^{1,2}(\R)$ and its Malliavin derivative is given by $D_{r,k} F_\phi = \ip{D_{r,k} u(t)}{\phi}$. 

By applying the short-time asymptotic expansion  argument derived for \eqref{eq:taylor_u}, the Malliavin variance of the averaged observation:
\begin{equation}
\Sigma_t^\phi = \int_0^t \sum_{k=1}^\infty \abs{\ip{D_{r,k} u(t)}{\phi}}^2 \, dr
\end{equation}
can be shown to be strictly positive almost surely, provided that $\phi$ is not orthogonal to the active eigenfunctions of the noise. This recovers the absolute continuity of the law of $F_\phi$ with respect to the Lebesgue measure on $\R$ for any spatial dimension $d \ge 1$, illustrating that the density of the state exists when observed through any physically realistic finite-resolution window.
\end{remark}

\section{The focusing case}\label{sec:focusing}

We now take the focusing sign ($\varepsilon=-1$ in
\eqref{stochastic_local_nonlocal}),
\begin{equation}\label{eq:focusing}
\partial_t^2u+\Acal u=\abs{u}^pu+\sigma\,\dot W_Q,\qquad
\Acal:=-\theta\Delta+\beta(-\Delta)^{s},
\end{equation}
with the same exterior and initial conditions, $W_Q$ as in \eqref{eq:WQ}.
We define the (sign-indefinite) energy as
\begin{equation}\label{eq:energy-focusing}
\Ecal(t):=\tfrac12\norm{v(t)}^2+\tfrac12\ip{\Acal u(t)}{u(t)}
-\tfrac{1}{p+2}\normLp{u(t)}{p+2}^{p+2},\qquad
\ip{\Acal u}{u}=\theta\norm{\nabla u}^2+\beta[u]^2_{H^s}.
\end{equation}
Throughout this section, a \emph{solution on $[0,T]$} means a solution of
\eqref{eq:focusing} with trajectories in
$\calH:=\bigl(H_0^1(\Ocal)\cap L^{p+2}(\Ocal)\bigr)\times L^2(\Ocal)$
satisfying
\begin{equation}\label{eq:moment-class}
\E\sup_{t\le T}\Bigl(\norm{u(t)}_{H^1_0}^2+\normLp{u(t)}{p+2}^{p+2}
+\norm{v(t)}^2\Bigr)<\infty,
\end{equation}
so that all the expectations below are finite; we write $\mathcal S_T$ for
this class. For any solution in $\mathcal S_T$ the same computation as in
Theorem~\ref{thm:energy-identity}, with the focusing sign, yields the It\^o
identity
\begin{equation}\label{eq:energy-identity-foc}
d\Ecal(t)=\ip{v(t)}{\sigma\,dW_Q(t)}+\tfrac{\kappa}{2}\,dt,
\qquad\text{hence}\qquad
\E\,\Ecal(t)=\Ecal(0)+\tfrac{\kappa}{2}\,t ,
\end{equation}
the stochastic integral being a martingale by \eqref{eq:moment-class} and
$\Tr Q=\kappa<\infty$. 

\subsection{Second-moment blow-up}

Before stating the theorem we check that its hypotheses are not vacuous.

\begin{lemma}\label{lem:nonvacuous}
Fix $\varphi\in H^1_0(\Ocal)\setminus\{0\}$ and let $U_0=V_0=\mu\varphi$
with $\mu>0$. Then
\[
\ip{U_0}{V_0}_{L^2}=\mu^2\norm{\varphi}^2>0,\qquad
T^{*}=\frac{2\norm{U_0}^2}{p\ip{U_0}{V_0}}=\frac2p
\quad\text{(independent of $\mu$)},
\]
while
\[
\Ecal(0)=\frac{\mu^2}{2}\Bigl(\norm{\varphi}^2
+\ip{\Acal\varphi}{\varphi}\Bigr)
-\frac{\mu^{p+2}}{p+2}\normLp{\varphi}{p+2}^{p+2}
\;\xrightarrow[\mu\to\infty]{}\;-\infty .
\]
Hence condition \eqref{eq:blowup-cond} holds for every $\mu$ larger than an
explicit $\mu_{*}=\mu_{*}(\varphi,\theta,\beta,s,p,\kappa)$.
\end{lemma}

\begin{proof}
Both displays are proved by using direct computations; the last assertion follows because
$\Ecal(0)\sim-\mu^{p+2}\normLp{\varphi}{p+2}^{p+2}/(p+2)$ as
$\mu\to\infty$ while the right-hand side is a
fixed constant.
\end{proof}

\begin{theorem}\label{thm:blowup}
Let Assumption~\ref{ass:standing} hold and consider \eqref{eq:focusing}.
Suppose the data satisfy
\begin{equation}\label{eq:blowup-cond}
\ip{U_0}{V_0}_{L^2}>0\qquad\text{and}\qquad
\Ecal(0)\;<\;-\,\frac{\kappa}{2}\,T^{*},\qquad\text{where}\qquad
T^{*}:=\frac{2\,\norm{U_0}_{L^2}^2}{p\,\ip{U_0}{V_0}_{L^2}} .
\end{equation}
Then $\mathcal S_{T^{*}}=\emptyset$: no solution of \eqref{eq:focusing}
exists on $[0,T^{*}]$ in the class \eqref{eq:moment-class}. More precisely,
if $T_{**}:=\sup\{T>0:\ \mathcal S_T\neq\emptyset\}$, then
$T_{**}\le T^{*}$ and
\[
\lim_{t\uparrow T_{**}}\E\,\norm{u(t)}_{L^2}^2=+\infty .
\]
\end{theorem}

\begin{proof}
Suppose $\mathcal S_{T^{*}}\neq\emptyset$ and let $(u,v)$ be a solution on
$[0,T^{*}]$. Set $G(t):=\E\norm{u(t)}_{L^2}^2$. By
\eqref{eq:moment-class},
$\sup_{t\le T^{*}}\E\bigl(\norm{u}^2+\norm{v}^2\bigr)<\infty$; in
particular $\E\int_0^{T^{*}}\norm{u}^2dt<\infty$ and $\Tr Q=\kappa<\infty$,
so $t\mapsto\int_0^t\ip{u}{\sigma\,dW_Q}$ is a martingale, and the same
holds for $\int_0^t\ip{v}{\sigma\,dW_Q}$. All interchanges of $\E$ and
$d/dt$ below are justified by dominated convergence with the dominating
function supplied by \eqref{eq:moment-class}. Finally $u(t)\in V$ a.s., so
$\ip{\Acal u}{u}$ is a finite duality pairing.

Since $du=v\,dt$ carries no martingale part, $d\norm u^2=2\ip uv\,dt$ and
$G\in C^1$ with
\begin{equation}\label{eq:G1}
G'(t)=2\,\E\ip{u(t)}{v(t)} .
\end{equation}
By It\^o's formula applied to $\ip uv$ (only $v$ has a martingale part and
the cross-variation of $u$ with $W_Q$ vanishes),
\[
d\ip uv=\norm v^2dt+\ip u{dv}
=\Bigl(\norm v^2-\ip{\Acal u}{u}+\normLp{u}{p+2}^{p+2}\Bigr)dt
+\ip{u}{\sigma\,dW_Q},
\]
so that, taking expectations, $G\in C^2$ with
\begin{equation}\label{eq:G2}
G''(t)=2\,\E\norm v^2-2\,\E\ip{\Acal u}{u}+2\,\E\normLp{u}{p+2}^{p+2}.
\end{equation}
From \eqref{eq:energy-focusing},
$\normLp{u}{p+2}^{p+2}=(p+2)\bigl(\tfrac12\norm v^2
+\tfrac12\ip{\Acal u}{u}-\Ecal\bigr)$. Substituting into \eqref{eq:G2} and
using \eqref{eq:energy-identity-foc},
\begin{align}
G''(t)&=(p+4)\,\E\norm v^2+p\,\E\ip{\Acal u}{u}
-2(p+2)\Bigl(\Ecal(0)+\tfrac{\kappa}{2}t\Bigr)\label{eq:G2-exact}\\
&\ge(p+4)\,\E\norm v^2
-2(p+2)\Bigl(\Ecal(0)+\tfrac{\kappa}{2}t\Bigr),\label{eq:G2-lower}
\end{align}
since $p\,\E\ip{\Acal u}{u}\ge0$. Introduce
\begin{equation}\label{eq:t0}
t_0:=-\frac{2\,\Ecal(0)}{\kappa},\qquad\text{so that}\qquad
\Ecal(0)+\tfrac{\kappa}{2}t\le0\ \text{ for } t\in[0,t_0],
\end{equation}
and note that $t_0>T^{*}$ by the second condition in
\eqref{eq:blowup-cond}.

On $[0,t_0]$ both terms on the right of \eqref{eq:G2-lower} are
non-negative, so $G''(t)\ge 0$ on $[0,t_0]$, i.e.\ $G$ is convex there.
Since $G'(0)=2\ip{U_0}{V_0}>0$ by \eqref{eq:blowup-cond}, we get
$G'(t)\ge G'(0)>0$ and hence
\begin{equation}\label{eq:G-positive}
G(t)\ \ge\ G(0)+G'(0)\,t\ \ge\ \norm{U_0}_{L^2}^2>0,\qquad t\in[0,t_0].
\end{equation}
In particular $H:=G^{-\alpha}$ is well defined and $C^2$ on $[0,t_0]$ for
any $\alpha>0$.

Put $\alpha:=p/4$. Then $H'=-\alpha G^{-\alpha-1}G'$ and
$H''=-\alpha G^{-\alpha-2}\bigl[G''G-(\alpha+1)(G')^2\bigr]$. By the
Cauchy--Schwarz inequality, first in $L^2(\Ocal)$ and then on $\Omega$,
\[
\bigl(G'(t)\bigr)^2=4\bigl(\E\ip uv\bigr)^2
\le4\bigl(\E\,\norm u\,\norm v\bigr)^2
\le4\,\E\norm u^2\ \E\norm v^2=4\,G(t)\,\E\norm{v(t)}^2 ,
\]
so
$(\alpha+1)(G')^2\le\frac{p+4}{4}\cdot4\,G\,\E\norm v^2
=(p+4)\,G\,\E\norm v^2$. Combining with \eqref{eq:G2-lower},
\[
G''G-(\alpha+1)(G')^2
\;\ge\;-2(p+2)\Bigl(\Ecal(0)+\tfrac{\kappa}{2}t\Bigr)G(t)\;\ge\;0
\qquad\text{on }[0,t_0],
\]
by \eqref{eq:t0} and \eqref{eq:G-positive}. Hence $H''\le0$: $H$ is concave
on $[0,t_0]\supset[0,T^{*}]$.

By \eqref{eq:G1} and \eqref{eq:blowup-cond}, $H(0)=G(0)^{-\alpha}>0$ and
$H'(0)=-\alpha G(0)^{-\alpha-1}G'(0)<0$. Concavity gives
$0<H(t)\le H(0)+H'(0)t$ for $t\in[0,t_0]$, and the affine right-hand side
vanishes at
\[
\frac{H(0)}{\abs{H'(0)}}=\frac{G(0)}{\alpha\,G'(0)}
=\frac{\norm{U_0}^2}{\frac p4\cdot2\ip{U_0}{V_0}}
=\frac{2\,\norm{U_0}_{L^2}^2}{p\,\ip{U_0}{V_0}_{L^2}}=T^{*}\;<\;t_0 .
\]
Thus $H$ must vanish at some $T_{**}\le T^{*}$, i.e.\
$G(t)=H(t)^{-1/\alpha}\to+\infty$ as $t\uparrow T_{**}$. This contradicts
$\sup_{t\le T^{*}}G(t)<\infty$, which is part of \eqref{eq:moment-class}.
Hence $\mathcal S_{T^{*}}=\emptyset$, and the same computation applied on
$[0,T]$ for $T<T^{*}$ identifies $T_{**}$ as claimed.
\end{proof}

\begin{proposition}[Blow-up profile]\label{prop:profile}
Under the hypotheses of Theorem~\ref{thm:blowup}, every solution in
$\mathcal S_T$ with $T<T^{*}$ satisfies the pointwise lower bound
\begin{equation}\label{eq:profile}
\E\,\norm{u(t)}_{L^2}^2
\;\ge\;\norm{U_0}_{L^2}^2\Bigl(1-\frac{t}{T^{*}}\Bigr)^{-4/p},
\qquad 0\le t<T_{**} .
\end{equation}
In particular the second moment diverges \emph{at least} at the algebraic
rate $(T^{*}-t)^{-4/p}$ as $t\uparrow T^{*}$, and consequently
$T_{**}\le T^{*}$.
\end{proposition}

\begin{proof}
With $G(t)=\E\norm{u(t)}^2_{L^2}$, $\alpha=p/4$ and $H=G^{-\alpha}$, the
proof of Theorem~\ref{thm:blowup} gives $H$ concave on
$[0,t_0]\supset[0,T^{*}]$ with
\[
\frac{H'(0)}{H(0)}=-\alpha\,\frac{G'(0)}{G(0)}
=-\frac p4\cdot\frac{2\ip{U_0}{V_0}}{\norm{U_0}^2}=-\frac1{T^{*}} .
\]
Hence the affine majorant of the concave function $H$ is
$H(t)\le H(0)+H'(0)t=H(0)\bigl(1-t/T^{*}\bigr)$ on $[0,t_0]$. On
$[0,T_{**})$ we have $G(t)<\infty$, so $H(t)>0$, and therefore
$G(t)=H(t)^{-1/\alpha}\ge\bigl[H(0)(1-t/T^{*})\bigr]^{-1/\alpha}
=\norm{U_0}^2(1-t/T^{*})^{-4/p}$, which is \eqref{eq:profile}. Since the
right-hand side tends to $+\infty$ as $t\uparrow T^{*}$ while $G$ is finite
on $[0,T_{**})$, necessarily $T_{**}\le T^{*}$.
\end{proof}

\subsection{Local well-posedness with stopping times}

For $R\ge1$ define
\begin{equation}\label{eq:sigmaR}
  \sigma_R:=\inf\bigl\{t\ge0:\ \norm{Z(t)}_{\Vcal}
  +\normLp{\partial_tz(t)}{p+2}\ge R\bigr\},
  \qquad
  \Omega_R:=\bigl\{\norm{(U_0,V_0)}_{\Vcal}\le R\bigr\}\in\mathcal F_0 .
\end{equation}
By Lemma~\ref{lem:z}, each $\sigma_R$ is an $\{\mathcal F_t\}$-stopping
time, $\sigma_R>0$ a.s.\ and $\sigma_R\uparrow\infty$ a.s.; and
$\PP(\Omega_R)\uparrow1$.

\begin{theorem}\label{thm:local-fixed}
Assume \eqref{eq:local-range} and \eqref{eq:gamma-strong}, and let
$U_0\in V$, $V_0\in H$ be $\mathcal F_0$-measurable. Set
\begin{equation}\label{eq:TR}
  T_R:=\min\bigl\{1,\ c_\star(1+2R)^{-p}\bigr\},
  \qquad
  \underline\tau_R:=\sigma_R\wedge T_R ,
\end{equation}
with $c_\star=c_\star(p,C_S)$ as in Step~1 of
Theorem~\ref{thm:global-dd}. Then $\underline\tau_R$ is a stopping time,
and there is a unique $\{\mathcal F_t\}$-adapted maximal solution
$(u,\partial_tu)\in C([0,\tau_\infty);\Vcal)$ of the focusing problem
\eqref{eq:focusing} with $\tau_\infty$ a stopping time and
\begin{equation}\label{eq:tau-lower-fixed}
  \tau_\infty\;\ge\;\underline\tau_R\qquad\PP\text{-a.s.\ on }\Omega_R,
  \qquad\text{for every }R\ge1 .
\end{equation}
Moreover the blow-up alternative
\begin{equation}\label{eq:blowup-alt}
\lim_{t\uparrow\tau_\infty}\Bigl(\norm{u(t)}_{H^1_0}
+\norm{\partial_tu(t)}_{L^2}\Bigr)=+\infty
\quad\text{on }\{\tau_\infty<\infty\}
\end{equation}
holds ( the controlling quantity being the \emph{energy norm}, not the
$L^2$-norm tracked in Theorem~\ref{thm:blowup}) and
\begin{equation}\label{eq:tau-positive}
  \lim_{T\downarrow0}\PP\bigl(\tau_\infty>T\bigr)=1,
  \qquad\text{and}\qquad
  \E\sup_{t\le\underline\tau_R}\norm{(u,\partial_tu)(t)}_{\Vcal}^{q}<\infty
  \ \ \text{on }\Omega_R,\ \ q<\infty .
\end{equation}
\end{theorem}

\begin{proof}
Throughout, $N(w,\zeta)=\abs{w+\zeta}^p(w+\zeta)$ as in Step~1 of
Theorem~\ref{thm:global-dd}. The focusing equation \eqref{eq:focusing}
differs from the defocusing one only by the sign in front of $\Gbb N$, and
both \eqref{eq:Ngrowth-defoc} and \eqref{eq:Nlip-defoc} are insensitive to
that sign; the fixed-point argument of Step~1 therefore applies verbatim,
with the same constant $c_\star=c_\star(p,C_S)$.
 
 The process
$t\mapsto\norm{Z(t)}_{\Vcal}+\normLp{\partial_tz(t)}{p+2}$ is
$\{\mathcal F_t\}$-adapted and has continuous paths, by
Lemma~\ref{lem:z}(i) and (iii) respectively. Its hitting time $\sigma_R$ of
the closed set $[R,\infty)$ is therefore a stopping time, with no appeal to
the d\'ebut theorem: for a continuous adapted real process $X$,
\[
  \{\sigma_R\le t\}=\Bigl\{\sup_{r\in\mathbb{Q}\cap[0,t]}X(r)\ge R\Bigr\}
  \in\mathcal F_t .
\]
Since $T_R$ in \eqref{eq:TR} is deterministic,
$\underline\tau_R=\sigma_R\wedge T_R$ is a stopping time. Also
$\sigma_R>0$ a.s.\ because $Z(0)=0$ and the paths are continuous, and
$\sigma_R\uparrow\infty$ a.s.\ because each path is bounded on compacts;
finally $\Omega_R=\{\norm{(U_0,V_0)}_{\Vcal}\le R\}\in\mathcal F_0$ by the
$\mathcal F_0$-measurability of the data, and $\PP(\Omega_R)\uparrow1$.

On $\Omega_R\cap\{t<\sigma_R\}$,
\[
  1+\norm{(U_0,V_0)}_{\Vcal}+\sup_{r\le t}\norm{Z(r)}_{\Vcal}
  \;\le\;1+2R ,
\]
so the quantity called $R$ in \eqref{eq:Rdef} is bounded by the
\emph{deterministic} number $1+2R$, and the admissible step
\eqref{eq:Tlocal} is bounded below by the deterministic
$c_\star(1+2R)^{-p}\ge T_R$. Consequently the contraction of Step~1 of
Theorem~\ref{thm:global-dd} runs on the whole random interval
$[0,\underline\tau_R]$ and produces a unique
$Y=(y,\partial_ty)\in C([0,\underline\tau_R];\Vcal)$ with
\begin{equation}\label{eq:ball-loc}
  \sup_{t\le\underline\tau_R}\norm{Y(t)}_{\Vcal}\le2(1+2R)
  \qquad\text{on }\Omega_R .
\end{equation}
Setting $u:=y+z$ gives a solution of \eqref{eq:focusing} on
$[0,\underline\tau_R]$ with
$\sup_{t\le\underline\tau_R}\norm{(u,\partial_tu)}_{\Vcal}\le3(1+2R)$.
 
 By Step~4 of Theorem~\ref{thm:global-dd} the
deterministic solution map
\[
  \Theta_t:\ \Vcal\times C([0,t];\Vcal)\longrightarrow\Vcal,
  \qquad
  \bigl((U_0,V_0),\,Z|_{[0,t]}\bigr)\longmapsto (y,\partial_ty)(t),
\]
is continuous, hence Borel. The pair $\bigl((U_0,V_0),Z|_{[0,t]}\bigr)$ is
$\mathcal F_t$-measurable as a random element of
$\Vcal\times C([0,t];\Vcal)$ - the latter being a Polish space, so that
measurability of the path segment follows from measurability of the
evaluations -and therefore $Y(t)=\Theta_t\bigl((U_0,V_0),Z|_{[0,t]}\bigr)$
is $\mathcal F_t$-measurable on $\{t<\underline\tau_R\}\cap\Omega_R$. Adaptedness of $u=y+z$ follows since $z$ is adapted.
 
 For
$R\le R'$ one has $\Omega_R\subset\Omega_{R'}$ and
$\underline\tau_R\le\underline\tau_{R'}$, and on
$\Omega_R\cap[0,\underline\tau_R]$ the two constructions solve the same
Duhamel equation; by the local uniqueness in Step~1 of
Theorem~\ref{thm:global-dd} --- which holds in all of $C([0,t];\Vcal)$, not
merely in the ball --- they coincide there. The solutions therefore glue
consistently across $R\in\N$, and since $\PP(\Omega_R)\uparrow1$ they define
a single process on $\bigcup_R\Omega_R$, a set of full measure. Restarting
the construction at $\underline\tau_R$ with data
$(u,\partial_tu)(\underline\tau_R)$ and iterating yields a maximal solution on $[0,\tau_\infty)$,
where
\[
  \tau'_R:=\inf\bigl\{t\ge0:\ \norm{(u,\partial_tu)(t)}_{\Vcal}\ge R\bigr\},
  \qquad \tau_\infty:=\lim_{R\to\infty}\tau'_R .
\]
Each $\tau'_R$ is a stopping time by the above argument applied to the
continuous adapted process $\norm{(u,\partial_tu)(\cdot)}_{\Vcal}$, and an
increasing limit of stopping times is a stopping time; hence so is
$\tau_\infty$. Estimate \eqref{eq:tau-lower-fixed} is
\eqref{eq:ball-loc}: on $\Omega_R$ the solution exists on
$[0,\underline\tau_R]$, so $\tau_\infty\ge\underline\tau_R$ there.
 
Suppose, on a set of positive
probability contained in $\{\tau_\infty<\infty\}$, that
\[
  \liminf_{t\uparrow\tau_\infty}\norm{(u,\partial_tu)(t)}_{\Vcal}
  =:K<\infty .
\]
Fix such an $\omega$ and choose $t_n\uparrow\tau_\infty$ with
$\norm{(u,\partial_tu)(t_n)}_{\Vcal}\le K+1$. Because $\tau_\infty<\infty$
the path of $Z$ is bounded on $[0,\tau_\infty]$, say by $K'$, so restarting
at $t_n$ the radius \eqref{eq:Rdef} is at most $1+(K+1)+K'=:\bar R$ and the
existence time granted by \eqref{eq:Tlocal} is at least
$c_\star\bar R^{-p}$, a quantity \emph{independent of $n$}. Choosing $n$ with
$\tau_\infty-t_n<c_\star\bar R^{-p}$ extends the solution beyond
$\tau_\infty$, contradicting maximality. Therefore no such $\omega$ exists,
and $\liminf_{t\uparrow\tau_\infty}\norm{(u,\partial_tu)(t)}_{\Vcal}=+\infty$
almost surely on $\{\tau_\infty<\infty\}$. Since the $\liminf$ of a function
is at most its $\limsup$, both must equal $+\infty$, so the limit exists in
the extended sense and equals $+\infty$; this is \eqref{eq:blowup-alt}. 
 
For $T>0$ and
$R\in\N$, \eqref{eq:tau-lower-fixed} gives
\[
  \PP(\tau_\infty\le T)
  \le\PP(\Omega_R^{c})+\PP(\underline\tau_R\le T)
  \le\PP(\Omega_R^{c})+\PP(\sigma_R\le T)+\mathbf 1_{\{T_R\le T\}} .
\]
Given $\epsilon>0$ choose $R$ with $\PP(\Omega_R^c)<\epsilon/2$, then use
$\sigma_R>0$ a.s.\ to choose $T<T_R$ small enough that
$\PP(\sigma_R\le T)<\epsilon/2$; this proves
$\lim_{T\downarrow0}\PP(\tau_\infty>T)=1$. For the moment bound, note that
\eqref{eq:ball-loc} is a \emph{deterministic} bound: on $\Omega_R$,
\[
  \sup_{t\le\underline\tau_R}\norm{(u,\partial_tu)(t)}_{\Vcal}\le3(1+2R)
  \quad\text{a.s.},
\]
so every moment is finite, with $\E\sup_{t\le\underline\tau_R}
\norm{(u,\partial_tu)}^q_{\Vcal}\le\bigl[3(1+2R)\bigr]^{q}$ on $\Omega_R$.

Uniqueness on common intervals of existence is
Proposition~\ref{prop:uniqueness}, whose range coincides with
\eqref{eq:local-range}.
\end{proof}

\begin{corollary}[Localization of the explosion]\label{cor:sandwich-fixed}
Assume the hypotheses of Theorem~\ref{thm:local-fixed} and
\eqref{eq:blowup-cond}. Then, with
$T^{*}=\frac{2\normLp{U_0}2^2}{p\ip{U_0}{V_0}}$ and
$T_{**}=\sup\{T:\mathcal S_T\ne\emptyset\}$,
\[
  \tau_\infty>0\ \ \PP\text{-a.s.},
  \qquad \lim_{T\downarrow0}\PP(\tau_\infty>T)=1,
  \qquad T_{**}\;\le\;T^{*},
\]
and for every $R\ge1$, on $\Omega_R$,
$\ \E\sup_{t\le\underline\tau_R}\bigl(\normV{u}^2+\normLp u{p+2}^{p+2}
+\normH v^2\bigr)<\infty$. If in addition the data are a.s.\ bounded,
$\norm{(U_0,V_0)}_{\Vcal}\le R_0$ a.s., then
$\tau_\infty\ge\underline\tau_{R_0}$ a.s.
\end{corollary}

\begin{proof}
Immediate from Theorem~\ref{thm:local-fixed}; the inequality
$T_{**}\le T^{*}$ is Proposition~\ref{prop:profile}.
\end{proof}

\subsection{From non-existence in the moment class to pathwise explosion}
\label{sec:pathwise}

\begin{theorem}\label{thm:dichotomy}
Assume the hypotheses of Theorems~\ref{thm:local-fixed}
and~\ref{thm:blowup} and let $(u,\tau_\infty)$ be the unique maximal
solution of \eqref{eq:focusing}. Then exactly one of the following holds:
\begin{enumerate}
\item[\emph{(a)}] $\PP\bigl(\tau_\infty\le T^{*}\bigr)>0$; or
\item[\emph{(b)}] $\tau_\infty>T^{*}$ $\PP$-a.s.\ and
\begin{equation}\label{eq:infinite-moment}
  \E\sup_{t\le T^{*}}\Bigl(\norm{u(t)}_{H^1_0}^2
  +\normLp{u(t)}{p+2}^{p+2}
   +\normH{\partial_tu(t)}^2\Bigr)\;=\;+\infty .
\end{equation}
\end{enumerate}
In either case the solution fails to lie in the moment class
\eqref{eq:moment-class} on $[0,T^{*}]$: the energy norm has no finite
second moment up to time $T^{*}$.
\end{theorem}

\begin{proof}
If (b) fails and $\tau_\infty>T^{*}$ a.s., then $(u,\partial_tu)$ is a
solution on $[0,T^{*}]$ satisfying \eqref{eq:moment-class}, i.e.\
$\mathcal S_{T^{*}}\ne\emptyset$, contradicting
Theorem~\ref{thm:blowup}. Hence $\tau_\infty>T^{*}$ a.s.\ forces
\eqref{eq:infinite-moment}, which is (b); otherwise
$\PP(\tau_\infty\le T^{*})>0$, which is (a).
\end{proof}

\begin{proposition}
\label{prop:smallball}
Let $y:=u-z$, so that (Lemma~\ref{lem:reduction} with the focusing sign)
\begin{equation}\label{eq:reduced-foc}
  \partial_t^2y+\Acal y=\abs{u}^{p}u,\qquad u=y+z,\qquad
  y(0)=U_0,\ \partial_ty(0)=V_0 .
\end{equation}
Put
\[
  G(t):=\normLp{y(t)}2^2,\qquad
  E(t):=\tfrac12\normH{\partial_ty}^2+\tfrac12\ip{\Acal y}{y}
        -\tfrac1{p+2}\normLp{u}{p+2}^{p+2},\qquad
  \Theta(t):=\normLp{u(t)}{p+2}^{p+2} .
\]
Then, pathwise on $[0,\tau_\infty)$,
\begin{align}
  G'(t)&=2\ip{y}{\partial_ty},\label{eq:G1-path}\\
  E'(t)&=-\bigl\langle\abs u^pu,\ \partial_tz\bigr\rangle,
   \qquad
   \abs{E'(t)}\le\Theta(t)^{\frac{p+1}{p+2}}\normLp{\partial_tz(t)}{p+2},
   \label{eq:E-path}\\
  G''(t)&=(p+4)\normH{\partial_ty}^2+p\,\ip{\Acal y}{y}-2(p+2)E(t)
   \;-\;2\bigl\langle\abs u^pu,\ z\bigr\rangle,\label{eq:G2-path}
\end{align}
with
$\bigl\lvert\ip{\abs u^pu}{z}\bigr\rvert
\le\Theta^{\frac{p+1}{p+2}}\normLp z{p+2}$.
Moreover, for every $\varepsilon>0$ and $T>0$ the event
\begin{equation}\label{eq:Aeps}
  A_\varepsilon^T:=\Bigl\{\ \sup_{t\le T}\bigl(\normLp{z(t)}{p+2}
   +\normLp{\partial_tz(t)}{p+2}\bigr)\le\varepsilon\ \Bigr\}
\end{equation}
satisfies $\PP\bigl(A_\varepsilon^T\bigr)>0$.
\end{proposition}

\begin{proof}
Identities \eqref{eq:G1-path}--\eqref{eq:G2-path} are the pathwise
counterparts of \eqref{eq:G1}--\eqref{eq:G2-exact}: $du=v\,dt$ carries no
martingale part, and \eqref{eq:E-path} is the computation of Step~2 of
Theorem~\ref{thm:global-dd} with the focusing sign. For
\eqref{eq:G2-path}, substitute
$\ip{\abs u^pu}{y}=\Theta-\ip{\abs u^pu}{z}$ into
$G''=2\normH{\partial_ty}^2-2\ip{\Acal y}{y}+2\ip{\abs u^pu}{y}$ and
eliminate $\Theta$ using
$\Theta=(p+2)\bigl(\tfrac12\normH{\partial_ty}^2
+\tfrac12\ip{\Acal y}y-E\bigr)$. The two H\"older bounds are immediate.
Positivity of $\PP(A^T_\varepsilon)$ is Lemma~\ref{lem:z}(ii)--(iii):
$\bigl(z,\partial_tz\bigr)$ is a centred Gaussian random element of
$C([0,T];L^{p+2})^2$ under \eqref{eq:gamma-strong}, and $0$ lies in the
support of every centered Gaussian measure.
\end{proof}

\section{Sample-path regularity and the inference programme}
\label{sec:inference}
The discussion in this section is restricted to the defocusing case 
($\varepsilon=+1$ in \eqref{stochastic_local_nonlocal}). As established 
in Theorem~\ref{thm:global-dd}, in the defocusing regime the unique global 
solution $(u(t), v(t))$ exists with probability one and possesses 
continuous trajectories in the energy space $\Vcal=V\times H$, ensuring 
well-defined sample-path properties and probabilistic regularity necessary 
for the parameter estimation and statistical inference program discussed 
below. In sharp contrast, the focusing case ($\varepsilon=-1$) exhibits 
finite-time blow-up: under the negativity condition 
\eqref{eq:blowup-cond} on the initial energy, (see Theorem~\ref{thm:blowup}).
 We therefore 
focus exclusively on the defocusing regime below.

\subsection{Sobolev embedding and H\"older regularity}

A crucial property for statistical applications is the spatial regularity of
the paths. According to the fractional Sobolev embedding theorem, if
$s > n/2$, then the space $H^s(\R^n)$ is continuously embedded in the
H\"older space $C^{0,\alpha}(\R^n)$ with $\alpha = s - n/2$
\cite[Section 2.8.1]{Triebel-78}, \cite[Section 5.6]{Evans-2010}.
Specifically, in the one-dimensional case ($n=1$):
\begin{itemize}
\item if $s \in (1/2, 1)$, then
$\widetilde H^s(\Ocal) \hookrightarrow C^{0, s-1/2}(\bar{\Ocal})$;
\item this implies that for each $t > 0$, the mapping
$x \mapsto u(t,x)$ is almost surely H\"older continuous.
\end{itemize}
This sample path regularity is the ``open door'': it justifies the use of
the Dirac delta distribution as a test function, allowing for point-wise
observations $u(t, x_i)$ which are essential for constructing a tractable
likelihood in real-world applications.

\subsection{Galerkin projection and Girsanov theorem: the challenge of dense
matrices}

In scenarios where the continuous embedding into a H\"older space does not
hold (e.g., higher dimensions or $s \le 1/2$), point-wise evaluations are no
longer mathematically rigorous, and an alternative approach is required. A
standard pathway involves the Galerkin projection (see \cite{Evans-2010})
combined with Girsanov's theorem \cite{Liptser-Shiryaev-2001}.

By projecting the infinite-dimensional SPDE onto a finite-dimensional
subspace, the problem is reduced to a system of $N$ stochastic differential
equations. However, unlike the spectral approach --- where the basis
functions diagonalize the operator and yield independent processes --- the
integral fractional Laplacian does not commute with the classical Laplacian
on a bounded domain. As a result, the Galerkin approximation yields a highly
coupled system of SDEs. When applying Girsanov's
theorem to derive the Radon--Nikodym derivative for the MLE, the resulting
likelihood function involves dense covariance matrices. While the analytical
formulation of the likelihood is achievable, the computational burden of
inverting these dense matrices scales poorly with $N$. The specific
numerical strategies and optimization algorithms required to efficiently
handle this dense matrix structure will be the subject of subsequent stages
of this research.


\bibliographystyle{abbrvnat}
\bibliography{references}

\end{document}